\documentclass{amsart}

\usepackage{amsfonts, amsmath, amsthm, amssymb, eufrak, epsfig, color,
bm, graphics, chngcntr, amscd}
\usepackage{psfrag}

\usepackage[all,cmtip]{xy}

\newcommand{\diff}{$\gamma_{\x,\y}$}
\newcommand{\difft}{$\tilde{\gamma}_{\x,\y}$}
\let\oldmarginpar\marginpar
\renewcommand\marginpar[1]{\-\oldmarginpar[\raggedleft\footnotesize #1]%
{\raggedright\footnotesize #1}}

\newcommand{\lbra}{{\em (}}
\newcommand{\rbra}{{\em )}}

\newcommand{\R}{\mathbb{R}}

\newcommand{\C}{\mathbb{C}}

\newcommand{\Z}{\mathbb{Z}}

\newcommand{\OneHalf}{\frac{1}{2}}

\newcommand{\cm}{\cdot}

\newcommand{\HD}{(\Sigma,\boldsymbol{\alpha},\boldsymbol{\beta})}

\newcommand\SpinC{\mathrm{Spin}^c}

\newcommand\Seif{R}

\newcommand\RelSpinC{\underline{\mathrm{Spin}}^c}
\newcommand\relspinc{\underline{\spinc}}

\newcommand\x{\mathbf x}

\newcommand\y{\mathbf y}

\newcommand\ModSphere{\ModFlow\left({\mathbb S}\longrightarrow 
\Sym^{g-1}(\Sigma_{1})\times \Sym^2(\Sigma_{2})\right)}
\newcommand\ModSpheres\ModSphere

\newcommand\UnparModSp{\widehat \ModSp}
\newcommand\UnparModFlow\UnparModSp
\newcommand\Mod\ModSp

\newcommand{\spinc}{\mathfrak s}

\newcommand\ModMaps{\mathcal M}
\newcommand\ModSp\ModMaps

\newcommand\Dom{\mathcal D}

\newcommand\spincrel\relspinc

\newcommand\HFK{HFK}

\newcommand\HFKa{\widehat\HFK}

\newcommand\Dual{\mathcal D}
\newcommand\Duality\Dual

\newcommand\ons{Ozsv{\'a}th and Szab{\'o}}
\newcommand\os{{Ozsv{\'a}th-Szab{\'o}}}

\newtheorem{thm}{Theorem}[section]          \newtheorem*{thm*}{Theorem}
\newtheorem{lem}[thm]{Lemma}                  \newtheorem*{lem*}{Lemma}
          \newtheorem*{ques*}{Question}
\newtheorem{prop}[thm]{Proposition}           \theoremstyle{definition}
\newtheorem{defn}[thm]{Definition}      \newtheorem*{defn*}{Definition}
\newtheorem{ex}[thm]{Example}                     \theoremstyle{remark}
\newtheorem{rem}[thm]{Remark}
\counterwithin{figure}{section}

\begin{document}

\title{On sutured Floer homology and the equivalence of Seifert surfaces}%
\author{Matthew Hedden, Andr\'as Juh\'asz, Sucharit Sarkar}%


\date{\today}%
\begin{abstract}
We study the sutured Floer homology invariants of the sutured manifold obtained by cutting a knot complement along a Seifert surface, $\Seif$.   We show that these invariants are finer than the ``top term" of the knot Floer homology, which they contain.  In particular, we use sutured Floer homology to distinguish two non-isotopic minimal genus Seifert surfaces for the knot $8_3$.   A key ingredient for this technique is finding  appropriate Heegaard diagrams for the sutured manifold associated to the complement of a Seifert surface.

\end{abstract}
\maketitle

\section{Introduction}
It is well-known that every knot in the three-sphere bounds an embedded orientable surface.  The various surfaces which a given knot bounds  are called {\em Seifert surfaces}, and play an important role in knot theory and low-dimensional topology as a whole.  Given a knot, $K$, the minimum genus of any Seifert surface for $K$ is called the {\em genus} of $K$.  The genus of a knot is a fundamental invariant, and minimal genus Seifert surfaces tell us a lot about the topological  and geometric properties of a knot.  In particular, the only knot of genus zero is the unknot.

A natural question is to what extent minimal genus Seifert surfaces are unique.   For instance, if a knot is fibered (that is, its complement is a fiber bundle over the circle with fibers consisting of Seifert surfaces) then the fiber surface is the unique minimal genus Seifert surface \cite{Burde}.  This means that any other minimal genus Seifert surface is isotopic to the fiber \footnote{There are two natural definitions of equivalence between Seifert surfaces, depending upon whether the isotopy occurs entirely in the complement.  We review these subtleties in Section \ref{sec:example}}.  Many examples are known, however, of knots with non-isotopic Seifert surfaces \cite{Alford,Eisner,Kakimizu1992,Kakimizu,Kob, Lyon}.

To date, most of the techniques for distinguishing Seifert surfaces have fallen into two categories: using the algebraic topology of  the surface's complement  e.g. the fundamental group \cite{Alford,Eisner,Lyon} or the Seifert form \cite{Trotter},  and Gabai's theory of sutured manifolds, \cite{Kob,Kakimizu1992}.  While the algebraic techniques are quite powerful, they often lead to difficult group or number theoretic questions and may be difficult to wield in general.   Additionally, some examples are beyond the scope of traditional algebraic topological tools.

Sutured manifold theory,  on the other hand,  has provided methods - but not invariants - which have been useful in understanding minimal genus Seifert surfaces for many knots.  The two techniques differ not only in spirit, but in the type of equivalence of Seifert surfaces which apply. Thus it would be desirable to have a computable invariant of Seifert surfaces which interacts well with the techniques of sutured manifold theory.

 The purpose of this article is to show that {\em sutured Floer homology}, introduced by the second author in \cite{sutured}, is precisely such an invariant.   Denoted $SFH(M,\gamma)$, the sutured Floer homology is an invariant associated to a (balanced) sutured manifold, $(M,\gamma)$, by a Lagrangian Floer homology construction.

Inspiration for the sutured Floer invariants came from the  invariants of knots and three-manifolds defined by \ons \   \cite{OSz,OSz3,OSz6,Ras}.  A key feature of the \os \  invariants  is their ability to detect  the genus of a knot, $K$.  The proof of this fact utilized sutured manifolds, but only so much as they were instrumental in providing taut foliations which allowed contact geometric and symplectic techniques to be employed \cite{OSz6}.


With the advent of sutured Floer homology, a precise relationship between Gabai's machinery and Heegaard Floer homology has now been established \cite{sutured,decomposition,suturedpolytope}.  Moreover, the genus detection  of knot Floer homology has an elegant reinterpretation in this theory which we briefly explain.

Given a Seifert surface, $\Seif$, for a knot $K\subset S^3$, we obtain a  sutured manifold, $S^3(\Seif)=(M,\gamma)$, by cutting along $\Seif$.  That is, we take $M = S^3 \setminus \mathrm{Int}(\Seif \times I)$ with suture $\gamma = \partial \Seif \times I.$ In \cite{decomposition} it was shown that $$SFH(S^3(\Seif)) \cong \widehat{HFK}(K,g(\Seif)).$$
Here, the right hand side is the knot Floer homology group of $K$ supported in Alexander grading $g(\Seif)$ \cite{OSz3}.  This isomorphism was then used, together with further properties of $SFH$ and results of Gabai, to reprove (among many other things) the fact that knot Floer homology detects the genus.    A striking aspect of this new proof is that it completely bypasses the four-dimensional methods which were originally used.

Now sutured Floer homology is an invariant of the sutured manifold, up to a natural notion of equivalence. It is immediate that isotopic Seifert surfaces produce equivalent sutured manifolds, and so one could hope that the sutured Floer homology of $S^3(\Seif)$ provides interesting information about the isotopy type of $\Seif$.    This optimism is quickly challenged by the isomorphism above; the knot Floer homology groups do not depend on the Seifert surface.  Examining the sutured Floer homology groups more closely, however, reveals structure not present in the knot Floer homology group.  This additional structure takes the form of a grading by (relative) $\SpinC$ structures.

For the reader unfamiliar with $\SpinC$ structures, we recall that the space of $\SpinC$ structures on a sutured manifold, $(M,\gamma)$, is isomorphic to $H_1(M;\Z)$ as an affine space.  Thus we can think of sutured Floer homology as having a grading by elements of $H_1(M;\Z)$.    Using this extra grading, we can provide the first explicit examples of minimal genus Seifert surfaces which are distinguished by sutured Floer homology.  Indeed, we have the following theorem

\begin{thm}
There exist two minimal genus Seifert surfaces, $\Seif_1$ and $\Seif_2$, for the knot $8_3$ for which no isotopy of $S^3$ sends $\Seif_1$ to $\Seif_2$.  Indeed, there does not exist an orientation-preserving diffeomorphism of the pairs $(S^3,\Seif_1)$, $(S^3,\Seif_2)$.
\end{thm}

We remark that while it was previously known that $\Seif_1$ and $\Seif_2$ are not isotopic {\em in the complement of $8_3$} \cite{Kob}, the question of whether they were isotopic was open.  Indeed, all previously available techniques fail to distinguish $\Seif_1$ and $\Seif_2$, up to isotopy.  See Subsection \ref{subsec:classical} for more details.  Using the above example, we also obtain

\begin{thm}For any $n\ge 1$, there exists a knot $K_n$ with Seifert surfaces $\{F_0,\dots F_n\}$, such that $F_i$ is not isotopic to $F_j$ for any $i\ne j$.
\end{thm}

Previously, there were examples known of  knots possessing infinitely many Seifert surfaces which are pairwise non-isotopic in the complement of $K$ \cite{Eisner}.   However, these examples are known to be isotopic in $S^3$, and again our theorem appears to be the strongest to date in the way of producing knots with many non-isotopic surfaces.

As the primary purpose of this article  is to provide a foundation for further study, the details of any particular example are somewhat beside the point.  We expect the techniques presented here to be applicable for a variety of questions in the study of Seifert surfaces, and conclude by briefly explaining the two major components of our framework.

The first is an explicit  understanding of  Heegaard diagrams for the sutured manifold associated to a Seifert surface.  Section \ref{sec:heegs} discusses these diagrams in detail.  In particular, we outline a very general method for obtaining such diagrams and then provide explicit algorithms.

 The second key feature is the extraction of  $\SpinC$ information from the aforementioned diagrams. The difference between any two $\SpinC$ structures  supporting sutured Floer homology yields an element of $H_1(M;\Z)$.  This element can be explicitly identified from the Heegaard diagrams.   However, it is difficult to determine whether $\gamma_1\in H_1(M_1;\Z)$, $\gamma_2\in H_1(M_2;\Z)$  presented by Heegaard diagrams for $M_1$, $M_2$, respectively, are identical (in the presence of an assumed equivalence between $M_1$ and $M_2$).  In the present context, the key observation is that   $H_1(S^3 \setminus \Seif \times I;\Z)\cong H_1(\Seif;\Z)$.  This isomorphism equips the former group with a bilinear form; namely, the Seifert form on $\Seif$.  We can use this form to distinguish elements of $H_1(S^3 \setminus \Seif_i \times I)$ obtained as   differences of $\SpinC$ structures supporting non-trivial Floer homology.  Distinguishing these elements, in turn, shows that the sutured manifolds are not equivalent and hence the Seifert surfaces are not isotopic.  We find this second feature particularly interesting, as this is the first instance that the Seifert form has made any real appearance in the context of Heegaard Floer homology.  
\\

\noindent {\bf Acknowledgment:} It is our pleasure to thank David Gabai, Chuck Livingston, and Zolt\'an Szab{\'o} for their interest in this work and many helpful conversations.

\section{Preliminaries}\label{sec:back}
Sutured manifolds were introduced by Gabai in \cite{Gabai}.  They provide a natural framework for constructing taut foliations on three-manifolds via inductive cut-and-paste procedures.  The motivation for taut foliations, in turn, is that they tell us about the Thurston norm of three-manifolds \cite{Thurston86}. In particular, they can be used to determine the genera of knots.  Sutured Floer homology is a generalization of \os \ Floer homology to an invariant of sutured manifolds, and was defined in \cite{sutured}.   Its definition and study were motivated by a desire to clarify and further explore connections between the \os \ invariants and Gabai's theory hinted at by the results in \cite{OSz6}.  In particular, a primary goal was to show that knot Floer homology detects fibered knots \cite{Ghiggini2007,NiFibered,decomposition}.

In this section, we begin by briefly recalling some basic notions from the theory of sutured manifolds.  We then discuss  sutured Floer homology, paying particular attention to sutured Heegaard diagrams.  These diagrams are the input for the sutured Floer homology invariants. Special focus will be given to sutured Heegaard diagrams adapted to a decomposing surface and the way in which decomposition of sutured manifolds is understood in terms of these diagrams.

We refer the reader to \cite{Gabai,Gabai3,Gabai6} for more details on sutured manifolds, and to \cite{sutured, decomposition} for details on sutured Floer homology.

\subsection{Sutured Manifolds}
The cornerstone of Gabai's machinery is the notion of a sutured manifold.
\begin{defn} \label{defn:1}
A   \emph{sutured  manifold}  $(M,\gamma)$   is  a   compact  oriented
3-manifold with  boundary, $(M,\partial M)$,  together with  a  set $\gamma  \subset
\partial  M$   of  pairwise  disjoint  annuli   $A(\gamma)$  and  tori
$T(\gamma).$   Furthermore,  the   interior  of   each   component  of
$A(\gamma)$ contains a \emph{suture}, i.e., a homologically nontrivial
oriented simple  closed curve. We denote  the union of  the sutures by
$s(\gamma).$

Finally,    every   component   of   $\Seif(\gamma)=\partial   M   \setminus
\text{Int}(\gamma)$    is required to be  oriented.    Define    $\Seif_+(\gamma)$   (resp.
$\Seif_-(\gamma)$)  to  be  those  components  of  $\partial  M  \setminus
\text{Int}(\gamma)$ whose normal vectors point out of (resp. into) $M$.  The
orientation  on   $\Seif(\gamma)$  must   be  coherent  with   respect  to
$s(\gamma),$ i.e., if $\delta$  is a component of $\partial \Seif(\gamma)$
and is  given the boundary  orientation, then $\delta$  must represent
the same homology class in $H_1(\gamma)$ as some suture.
\end{defn}

\begin{defn}
Two sutured manifolds $(M_1,\gamma_1), (M_2,\gamma_2)$ are said to be {\em equivalent} if there is an orientation-preserving diffeomorphism $f:M_1\rightarrow M_2$ which restricts to an orientation-preserving diffeomorphism between $\Seif(\gamma_1)$ and $\Seif(\gamma_2).$ \end{defn}

\begin{defn}
A sutured manifold $(M,\gamma)$ is  called \emph{balanced} if M has no
closed components,  $\chi(R_+(\gamma))=\chi(R_-(\gamma)),$ and the map
$\pi_0(A(\gamma)) \to \pi_0(\partial M)$ is surjective.
\end{defn}

The following two examples can be found in \cite{Gabai2}.

\begin{ex} \label{ex:1}
Let $\Seif$ be a compact oriented surface with no closed components.  Then
there is an  induced orientation on $\partial \Seif.$  Let $M=\Seif \times I,$
define  $\gamma  =\partial  \Seif  \times  I,$ and finally  put  $s(\gamma)  =
\partial \Seif \times \{1/2\}.$ The balanced sutured manifold $(M,\gamma)$
obtained  by  this  construction  is called  a  \emph{product  sutured
manifold}.
\end{ex}

 \begin{ex} \label{ex:2}
Let $Y$ be  a closed connected oriented 3-manifold  and let $\Seif \subset
Y$ be a compact oriented  surface with no closed components. We define a sutured manifold
$Y(\Seif)=(M,\gamma)$ to  be the sutured  manifold where $M =  Y \setminus
\text{Int}(\Seif \times  I),$ with the suture  $\gamma = \partial \Seif  \times I.$
 Furthermore $s(\gamma) = \partial \Seif \times \{1/2\}.$
\end{ex}

From the perspective of Floer homology, the following example is also quite relevant.

\begin{ex}\label{ex:knotexterior}
Let $K\subset Y$ be a knot, and let $Y_{2n}(K)=(M,\gamma_{2n})$ denote the sutured manifold with $M=Y\setminus\nu(K)$ the knot exterior, and $s(\gamma_{2n})$ consisting of $2n$ parallel copies of the meridian of $K$, with orientations alternating.
\end{ex}

The key to Gabai's inductive procedures is the concept of a sutured manifold  decomposition, which we now recall.  See  \cite[Definition   3.1]{Gabai}   and
\cite[Correction 0.3]{Gabai6}.  We begin with the notion of a decomposing surface.

\begin{defn} \label{defn:4}
Let $(M, \gamma)$ be  a sutured manifold. A \emph{decomposing surface}
is  an oriented, properly-embedded surface,  $S\subset M$,  such that  no
component  of  $\partial  S$  bounds  a disk  in  $\Seif(\gamma)$  and  no
component of  $S$ is a disk  $D$ with $\partial  D \subset R(\gamma).$
Moreover,  for every  component $\lambda$  of $S  \cap \gamma$  one of
(1)-(3) holds:
\begin{enumerate}
\item $\lambda$ is a  properly-embedded non-separating arc in $\gamma$
satisfying $|\lambda \cap s(\gamma)| = 1.$
\item $\lambda$ is  a simple closed curve in  an annular component $A$
of $\gamma$ in the same homology class as $A \cap s(\gamma).$
\item  $\lambda$  is  a  homotopically  non-trivial curve  in  a  torus
component $T$ of $\gamma,$ and  if $\delta$ is another component of $T
\cap S,$ then $\lambda$ and $\delta$ represent the same homology class
in $H_1(T).$
\end{enumerate}

\bigskip
\noindent A decomposing surface $\Seif$ defines a \emph{sutured manifold decomposition}, denoted

$$\xymatrix{(M, \gamma)\ar@{~>}[r]^-S &(M', \gamma')},$$

\noindent where $$M' = M  \setminus \text{Int}(N(S)),$$ $$\gamma' = (\gamma \cap
M') \cup N(S'_+ \cap R_-(\gamma)) \cup N(S'_- \cap R_+(\gamma)), $$
$$\Seif_+(\gamma')   =  ((R_+(\gamma)  \cap   M')  \cup   S'_+)  \setminus
\text{Int}(\gamma'),$$
$$\Seif_-(\gamma')   =  ((R_-(\gamma)  \cap   M')  \cup   S'_-)  \setminus
\text{Int}(\gamma').$$ Here $S'_+$ (resp.  $S'_-$) is the component of
$\partial N(S) \cap M'$ whose normal vector points out of (resp. into)
$M'.$
\end{defn}

\begin{rem}
In other words, the sutured  manifold $(M', \gamma')$ is constructed by
splitting $M$  along $S,$ creating $\Seif_+(\gamma')$ by  adding $S'_+$ to
what is  left of $\Seif_+(\gamma)$  and creating $\Seif_-(\gamma')$  by adding
$S'_-$  to what  is left  of  $\Seif_-(\gamma).$ Finally,  one creates  the
annuli of $\gamma'$ by ``thickening" $\Seif_+(\gamma') \cap R_-(\gamma').$
\end{rem}

The following lemma indicates that Examples \ref{ex:2} and \ref{ex:knotexterior} are connected by a sutured manifold decomposition.

\begin{lem} \label{lem:4}
Suppose  that $\Seif$ is  a Seifert  surface for  a knot  $K\subset Y.$  Then
$$\xymatrix{Y_{2n}(K)\ar@{~>}[r]^-R & Y(\Seif)}.$$
\end{lem}

\subsection{Sutured Floer homology}

We can associate to a balanced sutured manifold a collection of abelian groups, called the sutured Floer homology groups \cite{sutured}.  These groups are the homology groups of a chain complex, which is defined by a sutured  Heegaard  diagram.  Sutured Heegaard diagrams generalize Heegaard diagrams of closed 3-manifolds so that we can also describe sutured manifolds.

\begin{defn} \label{defn:2}
A   \emph{sutured   Heegaard   diagram}   is  a   tuple   $(   \Sigma,
\boldsymbol{\alpha}, \boldsymbol{\beta}),$ where $\Sigma$ is a compact
oriented   surface   with    boundary   and   $\boldsymbol{\alpha}   =
\{\,\alpha_1,   \dots,   \alpha_m\,\}$   and   $\boldsymbol{\beta}   =
\{\,\beta_1,  \dots, \beta_n\,\}$  are two  sets of  pairwise disjoint
simple closed curves in $\text{Int}(\Sigma).$
\end{defn}

Every  sutured   Heegaard  diagram  $(   \Sigma,  \boldsymbol{\alpha},
\boldsymbol{\beta})$ uniquely  defines a sutured  manifold $(M,
\gamma)$ using  the following construction: Let $M$  be the 3-manifold
obtained from  $\Sigma \times I$ by  attaching 3-dimensional 2-handles
along the  curves $\alpha_i \times  \{0\}$ and $\beta_j  \times \{1\}$
for $i=1,  \dots, m$ and $j=1,  \dots, n.$ The sutures  are defined by
taking  $\gamma =  \partial M  \times  I$ and  $s(\gamma)= \partial  M
\times \{1/2\}.$

\begin{defn} \label{defn:6}
A   sutured   Heegaard   diagram   $(   \Sigma,   \boldsymbol{\alpha},
\boldsymbol{\beta})$      is       called      \emph{balanced}      if
$|\boldsymbol{\alpha}|   =    |\boldsymbol{\beta}|$   and   the   maps
$\pi_0(\partial    \Sigma)   \to   \pi_0(\Sigma    \setminus   \bigcup
\boldsymbol{\alpha})$  and  $\pi_0(\partial  \Sigma) \to  \pi_0(\Sigma
\setminus \bigcup \boldsymbol{\beta})$ are surjective.
\end{defn}

The following is \cite[Proposition 2.14]{sutured}.

\begin{prop}
For  every  balanced  sutured  manifold $(M,\gamma)$  there  exists  a
balanced diagram defining it.
\end{prop}

In order to understand how $SFH$ behaves under surface decompositions, it is necessary to understand these operations at the level of Heegaard diagrams.  To this end, we have the following definition (Definition $4.3$ of \cite{decomposition}).

\begin{defn} \label{defn:16}
A balanced  diagram \emph{adapted} to  the decomposing surface  $\Seif$ in
$(M,\gamma)$    is   a   quadruple   $$(\Sigma,   \boldsymbol{\alpha},
\boldsymbol{\beta},  P),$$    satisfying  the  following  conditions.
\begin{enumerate}
\item $(\Sigma,  \boldsymbol{\alpha},  \boldsymbol{\beta})$  is  a  balanced
diagram  of  $(M,\gamma).$
\item $P  \subset  \Sigma$  is  a quasi-polygon (i.e., a closed subsurface of $\Sigma$ whose boundary is
a union of polygons) such that $P \cap \partial \Sigma$ is exactly the
set of vertices of $P.$
\item There is a decomposition $\partial P
= A  \cup B,$ where both $A$  and $B$ are unions  of pairwise disjoint
edges  of $P$ satisfying
$\alpha \cap B  = \emptyset$ and $\beta \cap A  = \emptyset$ for every
$\alpha \in  \boldsymbol{\alpha}$ and $\beta  \in \boldsymbol{\beta}.$
\item  $\Seif$  is obtained, up to  equivalence, by smoothing  the corners of
the surface $(P\times \{1/2\}) \cup  (A \times [1/2,1]) \cup (B \times
[0,1/2])  \subset  (M, \gamma)$  (recall the construction following Definition \ref{defn:2}).
\item The orientation of $\Seif$ is given  by the orientation of $P \subset \Sigma.$
\end{enumerate}
\end{defn}

We will frequently refer to a diagram adapted to $\Seif$ as a {\em surface diagram}.  A surface diagram allows us to represent decomposition along $\Seif$ in terms of Heegaard diagrams.  To describe this process, let  $(\Sigma,  \boldsymbol{\alpha},   \boldsymbol{\beta},  P)$  be  a
surface diagram for $\Seif$. To such a diagram, we can uniquely
associate a six-tuple $$D(P)  =  (\Sigma',  \boldsymbol{\alpha}',
\boldsymbol{\beta}',     P_A,    P_B,     p).$$ Here, $(\Sigma',
\boldsymbol{\alpha}', \boldsymbol{\beta}')$ is  a balanced diagram, $p
\colon  \Sigma' \to \Sigma$  is a  smooth map,  and $P_A,  P_B \subset
\Sigma'$ are two closed subsurfaces (see Figure \ref{fig:2}).  We will refer to $D(P)$ as the {\em decomposed diagram}.  $D(P)$ is constructed as follows.

We begin with $\Sigma'$. Let $P_A$
and $P_B$ be two disjoint copies of $P$, together with  diffeomorphisms $p_A \colon  P_A \to  P$ and
$p_B \colon  P_B \to  P.$   Then $$\Sigma' =  P_A \!\! \underset{ p_A^{-1}(A)\leftrightarrow A}\bigsqcup   \!\! \overline{(\Sigma\setminus P) }\!\! \underset{ p_B^{-1}(B)\leftrightarrow B}\bigsqcup \!\! P_B.$$
Thus, $\Sigma'$ is obtained by removing $P$ from $\Sigma$, and then gluing two copies of $P$ to the closure of the remaining surface, one copy glued along its $A$ edges and the other along its $B$ edges.

The  map $p \colon \Sigma' \to \Sigma$
agrees with  $p_A$ on $P_A$ and  $p_B$ on $P_B,$ and  it maps $\Sigma'
\setminus (P_A  \cup P_B)$ to  $\Sigma \setminus P$ using  the obvious
diffeomorphism.

 Finally,     let      $$\boldsymbol{\alpha}'     =
\{\,p^{-1}(\alpha)     \setminus     P_B     \colon     \alpha     \in
\boldsymbol{\alpha}\,\},$$   $$\boldsymbol{\beta}' = \{\,p^{-1}(\beta)
\setminus P_A \colon \beta \in \boldsymbol{\beta}\,\}.$$

Thus $p$ is  $1:1$ over $\Sigma \setminus P,$ is  $2:1$ over $P,$ and
$\alpha$ curves are lifted to $P_A$ and $\beta$ curves to $P_B.$  For the purposes of sutured Floer homology computations it is useful to note that, given a conformal structure on $\Sigma$, there is a unique conformal structure on $\Sigma'$ making $p$ into a conformal map.   The following proposition indicates that the decomposed diagram produces a Heegaard diagram for the sutured manifold obtained by decomposing along $\Seif$.

\begin{prop} \label{prop:1} \lbra\cite[Proposition 5.2]{decomposition}\rbra
\ Let $(M,\gamma)$  be a  balanced sutured manifold  and $$\xymatrix{(M,
\gamma)\ar@{~>}[r]^-S  &(M',\gamma')}$$  a  surface decomposition.  If
$(\Sigma,  \boldsymbol{\alpha},\boldsymbol{\beta},P)$   is  a  surface
diagram adapted to $S$ and if $$D(P) = (\Sigma', \boldsymbol{\alpha}',
\boldsymbol{\beta}',     P_A,     P_B,     p)$$   is the decomposed diagram, then  $(\Sigma',
\boldsymbol{\alpha}',  \boldsymbol{\beta}')$  is  a  balanced  diagram
defining $(M',\gamma').$
\end{prop}

\begin{figure}[htb]
\begin{center} \resizebox{200pt}{!}{\input{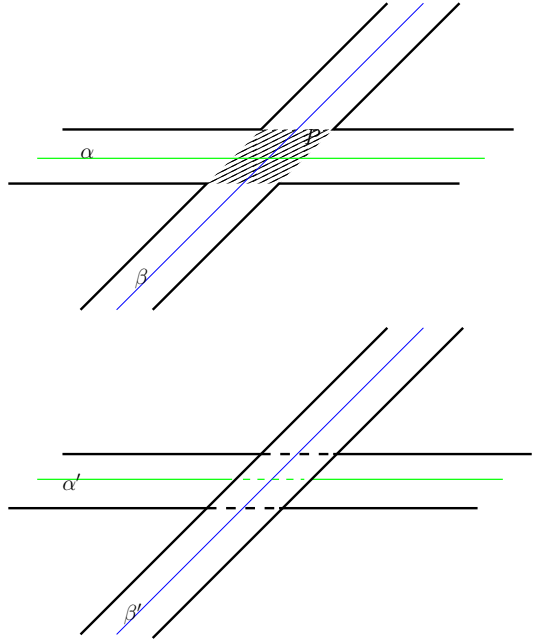}}
\end{center}
\caption{Balanced    diagram    before    and    after    a    surface
decomposition}\label{fig:2}
\end{figure}

\subsubsection{The sutured Floer chain complex}\label{subsec:SFH}
We conclude this section by briefly recalling the definition of the sutured Floer chain complex, and describing the splitting of this complex along relative $\SpinC$ structures.

Given a balanced sutured Heegaard diagram, $(\Sigma,\boldsymbol{\alpha},\boldsymbol{\beta})$, for a balanced sutured manifold, $(M,\gamma)$, one can define a chain complex $(C(\Sigma,\boldsymbol{\alpha},\boldsymbol{\beta}),\partial)$.
As a $\Z/2\Z$ vector space, $C(\Sigma,\boldsymbol{\alpha},\boldsymbol{\beta})$ is generated   by $k$-tuples $\x=x_1\times \dots \times x_k$ of intersection points, where $x_i\in \alpha_i\cap \beta_{\sigma(i)}$ (here, $\sigma$ is a permutation in the symmetric group on $k$ letters and $k=|\boldsymbol{\alpha}|=|\boldsymbol{\beta}|$ is the number of $\alpha$ curves).  If $k=0$, then despite having no curves we have a single generator (for the familiar reader, this is due to the fact that the $0$-th symmetric product of $\Sigma$ is a point, which coincides with the intersection of the two lagrangians).

The chain complex is equipped with a differential $\partial$ which counts points in moduli spaces of certain pseudo-holomorphic maps \cite{HolDisk,Lipshitz}.   To describe this, let us call the connected components of $\Sigma-\boldsymbol{\alpha}-\boldsymbol{\beta}$ {\em regions}, and denote them by $\Dom_1,\dots,\Dom_j.$  Given two generators $\x,\y \in C(\Sigma,\boldsymbol{\alpha},\boldsymbol{\beta})$ consider a linear combination of regions  $$\phi=\sum_{i=1}^{j} n_i \cm \Dom_i$$
which satisfies $\partial (\partial \phi|_\alpha)= \y-\x,$ i.e., the oriented boundary of the $\alpha$ components of $\partial\phi$ consists of the $k$-tuples of intersection points which comprise $-\x$ and $\y.$  If, furthermore, $\phi\cap \partial \Sigma = \emptyset$, we say that $\phi$ is a {\em domain} connecting $\x$ to $\y$.  Let us denote by $\pi_2(\x,\y)$ the set of domains connecting $\x$ to $\y$.

We define an endomorphism $\partial$ of $C(\Sigma,\boldsymbol{\alpha},\boldsymbol{\beta})$ by specifying it on generators:
$$\partial \x = \sum_{\y\in C(\Sigma,\boldsymbol{\alpha},\boldsymbol{\beta})} \sum_{\{\phi \in \pi_2(\x,\y) | \mu(\phi)=1\}} \#\widehat{\Mod}(\phi)\cm \y.
$$
In the formula, $\#\widehat{\Mod}(\phi)$ denotes the number (modulo $2$) of unparameterized pseudo-holomorphic maps of the unit disk $D^2\subset \C$ into the $k$-fold symmetric product of $\Sigma$, satisfying boundary conditions specified by $(\boldsymbol{\alpha},\boldsymbol{\beta},\x,\y)$ and whose homotopy class is determined by $\phi$.  The quantity $\mu(\phi)$ is the Maslov index of the domain, $\phi$, and the condition $\mu(\phi)=1$ is in place to ensure that the count can be performed (i.e., there exist only finitely many).   We refer the reader to \cite{sutured} for more details on the definition of $\partial$, but do call to mind the following important property (see  Lemma $3.2$ of \cite{HolDisk})

\begin{lem}\label{lem:positivity} Let $\phi=\sum_{i=1}^{j} n_i \cm \Dom_i$ be a domain. If $\#\widehat{\Mod}(\phi)\ne 0$, then $n_i\ge 0$ for all $i$.
\end{lem}

 For the purposes of computation, it is also useful to know that $\partial$ can be reformulated in terms of counting holomorphic maps of surfaces with boundary (and with marked points on the boundary) into $\Sigma\times D^2$.  This is made precise in \cite{Lipshitz}.

The following is contained in Theorems $7.1$ and $7.5$ of \cite{sutured}

\begin{thm}\label{thm:invariance} Let $(\Sigma,\boldsymbol{\alpha},\boldsymbol{\beta})$ be a Heegaard diagram for the sutured manifold, $(M,\gamma)$, and let
$(C(\Sigma,\boldsymbol{\alpha},\boldsymbol{\beta}),\partial)$ be as above. Then $\partial^2=0$. The resulting homology groups, denoted $SFH(M,\gamma)$, depend only on the equivalence class of sutured manifold.
\end{thm}

The above theorem suppresses some extra structure which we now discuss; namely, the splitting of sutured Floer homology into subgroups indexed by the set of {\em relative $\SpinC$ structures} on $(M,\gamma)$, which we denote $\RelSpinC(M,\gamma)$.  Indeed, to a generator $\x\in C(\Sigma,\boldsymbol{\alpha},\boldsymbol{\beta})$ one can associate a relative $\SpinC$ structure, $\spinc(\x)\in \RelSpinC(M,\gamma)$, as follows.

 First, pick a Morse function which determines the  Heegaard diagram and whose gradient vector field points  into $M$ along $\Seif_-(\gamma)$,  points out of $M$ along $\Seif_+(\gamma)$,  and which  is the gradient of the height function $s(\gamma)\times I \rightarrow I$ on $\gamma$.  Next, modify the gradient field in a neighborhood of flowlines specified by $x_i\in\x$.  This produces a non-vanishing vector field, $v$, with prescribed behavior on $\partial M$.  The homology class of $v$ (in the sense of Turaev \cite{Turaev}) specifies a relative $\SpinC$ structure, which we denote by $\spinc(\x)$.  The ``relative" terminology arises since we require vector fields to have prescribed behavior on $\partial M$. See Section $4$ of \cite{sutured} for more details.

For our purposes, the most important aspect of $\RelSpinC(M,\gamma)$ is that it is an affine set for $H^2(M,\partial M;\Z)$.  This implies, in particular, that we can talk about the  difference of two relative $\SpinC$ structures, $\spinc(\x)-\spinc(\y)\in H^2(M,\partial M;\Z)$.  Given two generators, $\x,\y$, we can concretely determine $\spinc(\x)-\spinc(\y)$ as follows.  First pick a collection of $k$ oriented sub-arcs of the $\alpha$ curves, $\gamma_{\alpha}\subset \boldsymbol{\alpha}$, which connect the intersection points $x_{i}$ to $y_{i}$.  Similarly, pick a collection of $k$ oriented sub-arcs of the $\beta$ curves, $\gamma_{\beta}\subset \boldsymbol{\beta}$, which connect the intersection points $y_{i}$ to $x_{\sigma(i)}$, for some permutation $\sigma$.  The sum, $\gamma_{\x,\y}=\gamma_{\alpha}+\gamma_{\beta}$, is a collection of oriented closed curves in $\Sigma\subset M$ whose homology class we denote by $\epsilon(\x,\y)\in H_1(M;\Z)$.  The following lemma is quite useful.

\begin{lem}{\label{lem:spincdiff}}(Lemma $4.7$ of \cite{sutured}) Let $\x,\y\in C(\Sigma,\boldsymbol{\alpha},\boldsymbol{\beta})$ be generators. Then $$\spinc(\x)-\spinc(\y) = \mathrm{PD}[\epsilon(\x,\y)]\in H^2(M,\partial M;\Z),$$
where PD$[\epsilon(\x,\y)]$ denotes the Poincar{\'e} dual of $\epsilon(\x,\y)$.
\end{lem}

The lemma makes clear the claim from the introduction; namely, that $(C(\Sigma,\boldsymbol{\alpha},\boldsymbol{\beta}),\partial)$ splits as a direct sum of complexes which are indexed by relative $\SpinC$ structures.  To see this, first observe that $(C(\Sigma,\boldsymbol{\alpha},\boldsymbol{\beta}),\partial)$ splits into subcomplexes corresponding  to the equivalence classes of the relation
$$ \x \sim \y  \Longleftrightarrow \pi_2(\x,\y)\ne \emptyset.$$
Next, note that if $\phi \in \pi_2(\x,\y) $ then $\epsilon(\x,\y)=[\partial \phi]=0 \in H_1(M;\Z)$.  Thus, if $\x$ and $\y$ are in the same subcomplex, they represent the same $\SpinC$ structure.  Conversely, if $\x$ and $\y$ represent the same $\SpinC$ structure, then $\epsilon(\x,\y)=0\in H_1(M;\Z)$.  In light of the isomorphism, $$H_1(M;\Z)\cong \frac{H_1(\Sigma;\Z)}{ \underset{i}{\mathrm{Span}}([\alpha_i]+[\beta_i])},$$
this implies that after possibly adding some copies of the $\alpha$ and $\beta$ curves to $\gamma_{\x,\y}$, we obtain a collection of curves which are null-homologous in $\Sigma$.   A null-homology  is an element $\phi\in \pi_2(\x,\y)$.

We have the following refinement of the theorem stated above:
\begin{thm}Let $(M,\gamma)$ be a sutured manifold.  Then $$SFH(M,\gamma)= \underset{\spinc\in\RelSpinC(M,\gamma)}\bigoplus SFH(M,\gamma,\spinc)$$  The homology group $SFH(M,\gamma,\spinc)$, depends only on the equivalence class of the sutured manifold and the relative $\SpinC$ structure $\spinc$.
\end{thm}

One of the most important aspects of sutured Floer homology is its behavior under surface decompositions, which we now describe.  We will need a definition. As above, suppose we have a decomposition $$\xymatrix{(M,
\gamma)\ar@{~>}[r]^-S  &(M',\gamma')}.$$  Let $(\Sigma,\boldsymbol{\alpha},\boldsymbol{\beta},P)$ be a surface diagram for  $S$.  Denote by $\mathcal{O}_P \subset \mathbb{T}_{\alpha} \cap \mathbb{T}_{\beta}$ the subset of generators, none of whose intersection points $x_i\in \x$ are contained in the quasipolygon, $P\subset \Sigma$.  Call such generators {\em outer generators}. The {\em outer complex} $C(\mathcal{O}_P)\subset C(\Sigma,\boldsymbol{\alpha},\boldsymbol{\beta})$ is the subcomplex generated by $\mathcal{O}_P.$ Finally, let $C(\Sigma',\boldsymbol{\alpha}',\boldsymbol{\beta}')$ be the chain complex associated to the decomposed diagram.
 The main result of \cite{decomposition} is the following.

 \begin{thm}\label{thm:decomposition}  The outer complex, $C(\mathcal{O}_P) \subset C(\Sigma,\boldsymbol{\alpha},\boldsymbol{\beta}),$ forms a subcomplex.  Moreover, the homology of this subcomplex is isomorphic to the homology of $C(\Sigma',\boldsymbol{\alpha}',\boldsymbol{\beta}')$.  In particular,
 $$SFH(M',\gamma'):=H_*(C(\Sigma',\boldsymbol{\alpha}',\boldsymbol{\beta}'))\cong H_*( C(\mathcal{O}_P))\le SFH(M,\gamma),$$
\noindent where $\le$ means ``direct summand".
\end{thm}

\begin{proof} (sketch)
To see that $C(\mathcal{O}_P)$ forms a subcomplex, consider a generator $\y\in (\mathbb{T}_{\alpha} \cap \mathbb{T}_{\beta}) \setminus \mathcal{O}_P.$  We will show that $\pi_2(\x,\y)=\emptyset$ for every $\x\in \mathcal{O}_P.$  This implies that $\y \notin \partial \x $ for every $\x \in \mathcal{O}_P$, and repeating with each  $\y\in (\mathbb{T}_{\alpha} \cap \mathbb{T}_{\beta}) \setminus \mathcal{O}_P$ we see that $\partial C(\mathcal{O}_P)\subset C(\mathcal{O}_P),$ i.e., $C(\mathcal{O}_P)$ is a subcomplex.

To see that $\pi_2(\x,\y)=\emptyset$ for $\x,\y$ as above, consider the collection of curves $\gamma_{\x,\y}=\gamma_\alpha+\gamma_\beta$ connecting $\x$ to $\y$.  Pushing  $\gamma_\alpha$ into the $\alpha$ handlebody and $\gamma_\beta$ into the $\beta$ handlebody, we obtain an oriented collection of curves, $\tilde{\gamma}_{\x,\y}\subset M$.  Note that since $\gamma_\alpha$ is oriented from $x_i$ to $y_i$, each intersection of $\tilde{\gamma}_{\x,\y}$ with the quasipolygon $P$ is positive.  Since the only intersections  $\tilde{\gamma}_{\x,\y}\cap S$ occur in $P$, this shows that $\#_{alg}(\tilde{\gamma}_{\x,\y}\cap S)>0$.  In particular $\epsilon(\x,\y)=[\tilde{\gamma}_{\x,\y}]\ne 0$, showing that $\pi_2(\x,\y)=\emptyset$.

Now it is immediate from the construction of the decomposed diagram that generators of $C(\Sigma',\boldsymbol{\alpha}',\boldsymbol{\beta}')$ are in bijection with $\mathcal{O}_P.$  Indeed, since $\alpha$ and $\beta$ arcs in $P$ lift to  $P_A$ and $P_B$ in the decomposed diagram, respectively, no intersection point $x_i\in P\subset \Sigma$ will lift to an intersection point in $\Sigma'$ (since $P_A\cap P_B=\emptyset$).  Hence any generator $\x$ containing $x_i\in P$ will not lift to a generator for $C(\Sigma',\boldsymbol{\alpha}',\boldsymbol{\beta}')$.  On the other hand, the decomposed diagram is identical to the surface diagram (before decomposition) outside of $P$.  Thus any outer generator lifts to a generator in $C(\Sigma',\boldsymbol{\alpha}',\boldsymbol{\beta}')$.

The most challenging part of the the proof arises when showing that the differential on $C(\Sigma',\boldsymbol{\alpha}',\boldsymbol{\beta}')$ is identical to the differential on $C(\mathcal{O}_P)$ which it inherits as a subcomplex of $C(\Sigma,\boldsymbol{\alpha},\boldsymbol{\beta})$.   To prove this, \cite{decomposition} adapts the algorithm of \cite{Sucharit} for computing Heegaard Floer homology to the context of sutured Floer homology.  By making the surface diagram  ``nice", the count of pseudo-holomorphic curves for each domain $\phi\in \pi_2(\x,\y)$ with $\mu(\phi)=1$ can be done explicitly using the Riemann mapping theorem, together with the fact that pseudo-holomorphic submanifolds of symplectic manifolds intersect positively.  Moreover, for a nice enough surface diagram, the decomposed diagram will also be nice and one can explicitly identify the differentials for the respective complexes.  See \cite{decomposition} for more details.
\end{proof}

As a corollary, one obtains the theorem mentioned in the introduction.

\begin{thm}\label{thm:topterm}
Let $\Seif$ be a Seifert surface for a knot $K\subset S^3$.  Then
 $$SFH(S^3(\Seif)) \cong \widehat{HFK}(K,g(\Seif)),$$
where the right hand side is the knot Floer homology group of $K$ supported in Alexander grading $g(\Seif)$ \cite{OSz3}.
\end{thm}
\begin{proof} (sketch) By Lemma \ref{lem:4} if we decompose $S^3_2(K)$ along $R$ we get $S^3(R).$ Let $(\Sigma,\boldsymbol{\alpha},\boldsymbol{\beta},P)$ be a surface diagram adapted to $R.$ The Alexander grading of a generator $\x\in C(\Sigma,\boldsymbol{\alpha},\boldsymbol{\beta}) \cong \widehat{CFK}(K)$ can be defined as
$$ \OneHalf \langle c_1(\spinc(\x)),[\Seif,\partial \Seif] \rangle, $$
where $c_1(\spinc(\x))\in H^2(S^3\setminus N(K), \partial;\Z)$ is the relative Chern class of a relative $\SpinC$ structure associated to $\x$ and $[\Seif,\partial \Seif]\in  H_2(S^3\setminus N(K), \partial;\Z)$ is the homology class of the surface.  This evaluation, in turn, can be computed as $$\chi(\Seif)-1+ \#\{x_i\in \x  | x_i\in P\},$$
where $P$ is the quasi-polygon representing $\Seif$ (see the proof of Theorem $5.1$ of \cite{OSz3} for motivation of this formula and \cite{decomposition} for precise details).    Together with Theorem \ref{thm:decomposition}, this shows that
  $$SFH(S^3(\Seif)) \cong \widehat{HFK}(K,-g(\Seif)).$$
However, $ \widehat{HFK}(K,g(\Seif))\cong \widehat{HFK}(K,-g(\Seif))$ by Proposition $3.10$ of \cite{OSz3}.
\end{proof}

\section{Constructing Heegaard diagrams adapted to a Seifert surface}
\label{sec:heegs}
Given a a Seifert  surface $\Seif$ for a knot $K\subset Y$,  we  wish to compute the sutured Floer homology groups, $SFH(Y(\Seif))$.  Since these groups are the homology of a chain complex associated to a balanced sutured Heegaard diagram  for  $Y(\Seif)$, it is necessary to produce such a diagram.  To do this, recall that $Y(\Seif)$ is obtained from the knot complement $Y_{2n}(K)$ by decomposition along $\Seif$.   According to Section \ref{sec:back}, it suffices to produce a balanced diagram for  $Y_{2n}(K)$ which is adapted to $\Seif$.  From there it is simple to obtain the decomposed  diagram.  This section will be dedicated to producing surface diagrams and clarifying the  decomposed diagram which, by Proposition \ref{prop:1},  will necessarily be adapted to $Y(\Seif)$.  Throughout we assume the genus of $\Seif$ to be $g$.

A surface diagram for $\Seif$ is, by definition, a Heegaard diagram for the sutured manifold associated to the knot complement which contains $\Seif$ as a quasi-polygon.  Building this diagram requires two main  steps:

\begin{enumerate}

\item Construct a Heegaard diagram $(\Sigma,\boldsymbol{\alpha},\boldsymbol{\beta} )$ for  the closed three-manifold, $Y$, which contains $\Seif$ as proper subsurface, $\Seif\subset \Sigma$.
\item Remove disks from $\Sigma$ along $\partial \Seif$, and modify the diagram by a sequence of isotopies 	and/or stabilizations to ensure that the diagram specifies $Y_{2n}(K)$ and is adapted to $\Seif$.
\end{enumerate}

For any given Seifert surface, there are many ways to perform each step so we remain intentionally vague for the moment.  The next two subsections discuss each step in detail. Indeed, for both steps we treat the case of a Seifert surface presented in an arbitrary manner.  This has the advantage of being completely general and should thus be useful in a variety of situations.

For the sake of clarity, the third subsection presents explicit diagrams for the case  of knots in $S^3$ with Seifert surfaces  presented in a particularly appealing form. The presentation is analogous to a knot projection, and the surface diagrams which we produce can be viewed as the analogue of the diagrams used by \ons \ in \cite{OSz9} which connect the knot Floer homology chain complexes to Kauffman states.


\subsection{Constructing a diagram for $Y$ containing $\Seif$}\label{subsec:surfaceinheeg}
Let $\Seif\subset Y$ be an embedded surface with non-empty boundary.  We now describe the construction of a  Heegaard diagram $\HD$ which contains $\Seif$ as an embedded, proper subsurface of $\Sigma$.
Similar diagrams have been useful in other contexts \cite{OSz3,OSz8,Yi}

\bigskip

Begin with $\Seif$.  Now thicken to obtain $\Seif\times I$.  This is a handlebody of genus $2g$.  We can  represent a basis for $H_1(\Seif,\partial \Seif;\Z)$ by $2g$ pairwise disjoint properly embedded arcs, $\gamma_i$, $i=1,\dots, 2g$.    Observe that $\gamma_i\times I$ is a properly embedded disk in the handlebody $\Seif\times I$.  Let $\beta_i=\partial(\gamma_i\times I)$.   Note that $\partial (\Seif\times I)$ consists of two copies of $\Seif$,  glued along $\partial \Seif$.   Thus $\partial(\Seif\times I)$ clearly contains $\Seif$ as an embedded proper subsurface.

Now $Y\setminus\Seif\times I$ is not necessarily a handlebody (indeed, it will be a handlebody precisely when $\Seif$ is a so-called {\em free} Seifert surface).  By adding a collection of three-dimensional $1$-handles, $\{h_i\}_{i=1}^k$  to $\Seif \times I$ we can ensure that $$H_\alpha = Y\setminus(\Seif\times I \cup h_1 \cup \dots \cup h_k)$$ is a handlebody (of genus $2g+k$).  Without loss of generality, we may assume that the feet of the $1$-handles lie on $\Seif\times \{0\}$, so that $\Seif=\Seif\times \{1\}\subset \partial H_\alpha$ is still embedded. Let $\beta_i$ denote the belt spheres of the $1$-handles, $i=2g+1,\dots,2g+k$.

Finally, pick a collection $\{\alpha_i\}_{i=1}^{2g+k}$ of linearly independent curves on $\Sigma_{2g+k}=\partial H_\alpha$ which bound disks in $H_\alpha$. Then $$(\Sigma_{2g+k}, \{\alpha_1,\dots,\alpha_{2g+k}\}, \{\beta_1,\dots,\beta_{2g+k}\})$$
is the desired Heegaard diagram.

\begin{rem} \label{rem:surfacecomp} Note that, by construction, the Heegaard diagram $$(\Sigma_{2g+k}, \{\alpha_1,\dots,\alpha_{2g+k}\}, \{\beta_{2g+1},\dots,\beta_{2g+k}\})$$ \noindent (i.e., the Heegaard diagram of the lemma without the first $2g$ $\beta$ curves)
specifies $Y\setminus\Seif\times I$.
\end{rem}

\subsection{Turning the diagram into a surface diagram}\label{subsec:algorithm}

Having produced a Heegaard diagram for $Y$ containing  $\Seif$, we now describe how to turn this into a surface diagram for $\Seif$.

Begin with the Heegaard diagram of Subsection \ref{subsec:surfaceinheeg}.  The desired result is achieved through the following algorithm (see Figure \ref{fig:algorithm} for a depiction of the algorithm):
\begin{enumerate}
\item Pick two points $\{z,w\} \subset K=\partial \Seif$.  The points divide $K$ into two arcs, which we label by  $A$ and $B$. Remove  disc neighborhoods $\{D(z),D(w)\}\subset \Sigma$  of the two points and call the resulting surface-with-boundary $\Sigma$.

If $A\cap \beta_i=\emptyset$, $B\cap \alpha_i=\emptyset$ for all $i$ then we are done.  If not, proceed to Step $2$.

\item Without loss of generality, assume $p\in A\cap \beta_i$ (if $p\in B\cap \alpha_i$, the same process applies with the roles of $A$ and $B$, $\alpha$ and $\beta$ reversed). Remove $p$ by one of the following operations.
\begin{itemize}
\item {\bf Isotopy:} Starting from $p$, perform a finger move of $\beta_i$  along the $A$ arc until a boundary component of $\Sigma$ is reached.  If other $\beta$ curves are encountered, perform the finger move to these curves as well to ensure that no intersections among $\beta$  curves are created.  Handleslide $\beta_i$, along with any other curves picked up by the isotopy, over the boundary component of $\Sigma$.

\item {\bf Stabilization:} Choose a  sub-arc, $\gamma$, of the $A$ arc containing $p$, which satisfies $\gamma\cap \alpha_i=\emptyset$ for all $i$.   Subdivide $K$ so that $\gamma$ is labeled $B$ and the two arcs adjacent to $\gamma$ are labeled $A$.  Let $D(\gamma)$ be a neighborhood of $\gamma$.  Adjoin $\partial(D(\gamma))$ to the collection of $\alpha$ curves.  Similarly, pick one of the two $A$ arcs adjacent to $\gamma$, and adjoin the boundary of its neighborhood to the $\beta$ curves. Finally, remove neighborhoods of the endpoints of $\gamma$ from $\Sigma.$
\end{itemize}

\item If $A\cap \beta_i=\emptyset$, $B\cap \alpha_i=\emptyset$ for all $i$ then we are done.  If not, repeat Step $2$.
\end{enumerate}

\begin{prop}  The above algorithm terminates at a surface diagram adapted to $\Seif$.
\end{prop}
\begin{proof}
As there are only a finite number of points in the initial set $\{A\cap \beta_i,B\cap \alpha_i\}$, it is clear that the above algorithm terminates.   To see that we have produced a surface diagram, first observe that the algorithm can be reinterpreted as an algorithm to convert the original diagram into a multi-pointed Heegaard diagram for $K\subset Y$. Indeed, instead of removing neighborhoods of $\{z,w\}$ and the endpoints of $\gamma$ in Steps $1$ and $2$, respectively, we could simply keep track of these points as pairs $\{z_i,w_i\}$ (and labeling  so that $z$'s are always the initial point of some $A$ arc, oriented by the orientation of $K$).  The resulting multi-pointed Heegaard diagram is then adapted to $K$, in the sense of Definition $2.1$ of \cite{MOS}.  Generalizing Example $2.4$ and Proposition $9.2$  of \cite{sutured} to the case of multi-pointed diagrams shows that removing neighborhoods of the basepoints produces a balanced diagram adapted $Y_{2n}(K)$\footnote{
 Note that $n$ is the number of stabilizations performed in the algorithm +1.}.

Finally, observe that since the original diagram contained $\Seif$ as a proper subsurface, the terminal diagram contains $\Seif$ as a quasi-polygon of the desired form.  Indeed, the $A$  and $B$ edges of the quasi-polygon are the $A$ and $B$ arcs produced by the algorithm.  Strictly speaking, we must make a local modification to these arcs as specified by Figure \ref{fig:modify} to ensure that $\partial \Seif$ is of the appropriate form.

\end{proof}

\begin{figure}
 \psfrag{A}{B}
 \psfrag{B}{A}
  \psfrag{R}{$R$}
  \psfrag{a}{$\beta$}
  \psfrag{b}{$\alpha$}
  \psfrag{p}{$p$}
  \psfrag{z}{$w$}
  \psfrag{w}{$z$}
  \psfrag{re}{Remove disks}
  \psfrag{is}{Isotopy}
  \psfrag{stab}{ Stabilization}
  \psfrag{1}{$(1)$}
   \psfrag{2}{$(2)$}
   \psfrag{Ks}{$K$}

\begin{center}
\includegraphics[width=450pt]{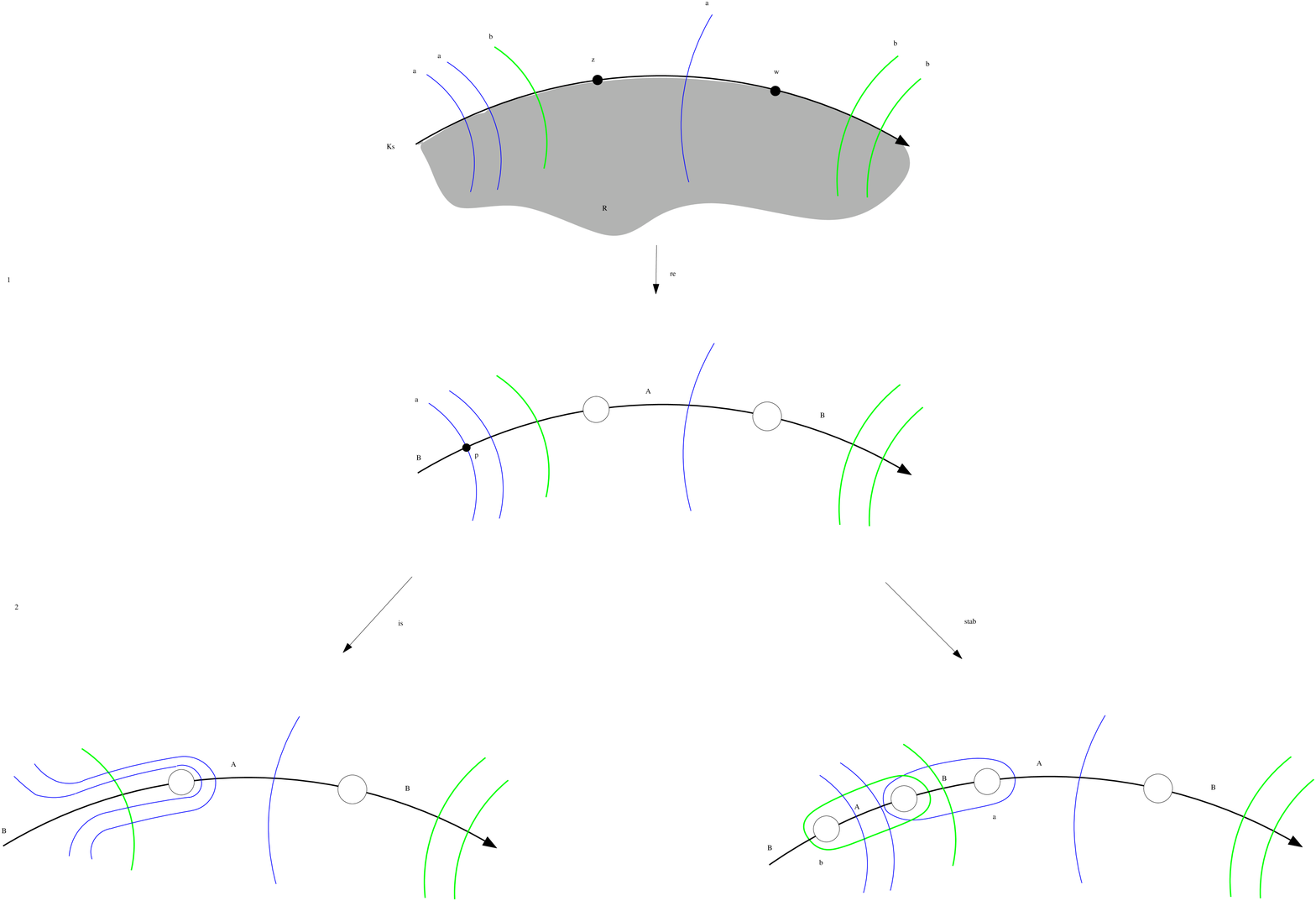}
\caption{\label{fig:algorithm}}
\end{center}
\end{figure}

\begin{figure}
 \psfrag{A}{$A$}
 \psfrag{B}{$B$}
  \psfrag{partial}{\small{$\partial \Sigma$}}
\begin{center}
\includegraphics[width=250pt]{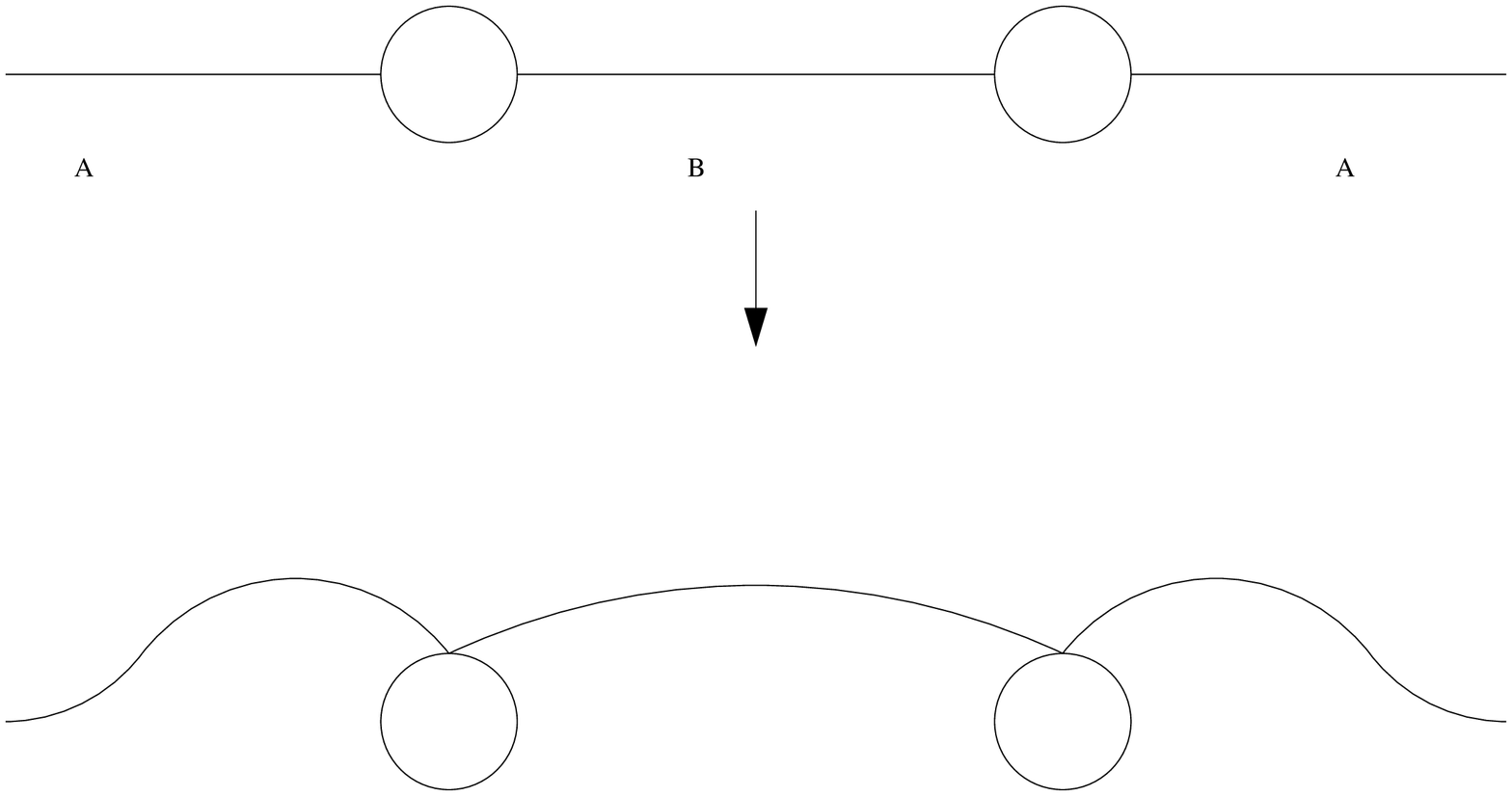}
\caption{\label{fig:modify}}
\end{center}
\end{figure}

\begin{figure}
\psfrag{1}{$1$}
\psfrag{a1}{$\alpha_1$}
\psfrag{a2}{$\alpha_2$}
\psfrag{a3}{$\alpha_3$}
\psfrag{b1}{$\beta_1$}
\psfrag{b2}{$\beta_2$}
\psfrag{b3}{$\beta_3$}

\begin{center}
\includegraphics[width=450pt]{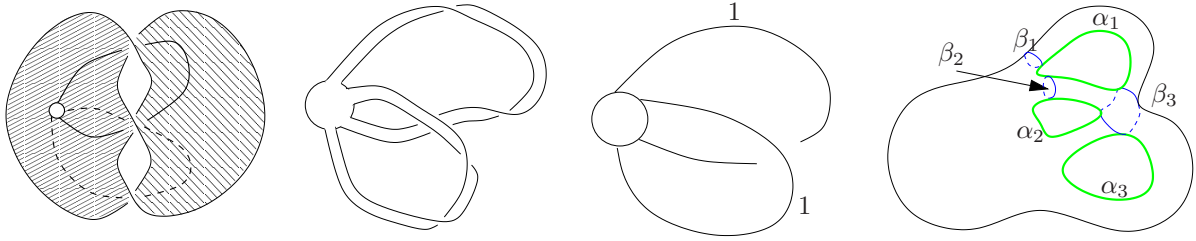}
\end{center}
\caption[section]{Starting from a Seifert surface for the trefoil, we first retract to a neighborhood of a one skeleton possessing a single $0$-cell.  We then encode this surface with a planar diagram, where $1$'s indicate that the corresponding bands have a full right-handed twist.  The third shows the handlebody which results from thickening the diagram.  The thick green curves are $\alpha$-curves specifying the complementary handlebody.  Each band contributes a thin blue $\beta$ curve.  The final $\beta$ curve comes from the crossing in the diagram, according to Figure \ref{fig:crossing} below.}\label{figureone}
\end{figure}

 \begin{figure}
\begin{center}
\includegraphics{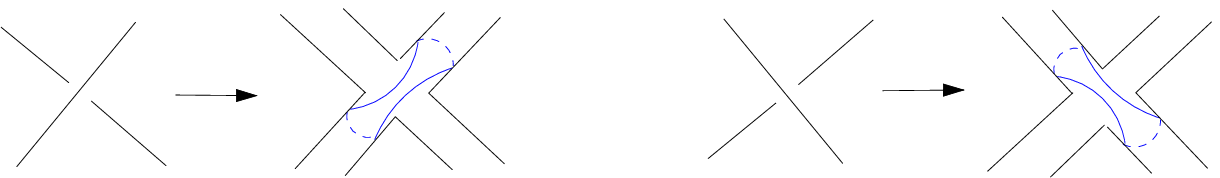}
\caption{\label{fig:crossing}}
\end{center}
\end{figure}

\begin{figure}
\psfrag{h}{Handles}
\psfrag{s}{Sphere}
\psfrag{k}{$K$}
\psfrag{p1}{$p_1$}
\psfrag{p2}{$p_2$}
\psfrag{p3}{$p_3$}
\psfrag{p4}{$p_4$}
\begin{center}
\includegraphics[width=260pt]{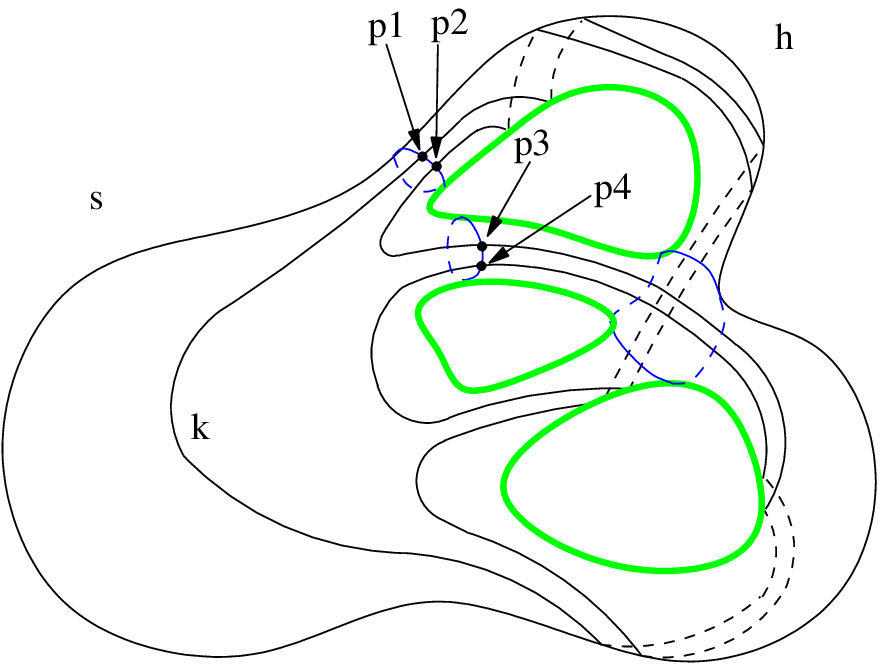}
\caption{\label{fig:trefoilwithknot}}
\end{center}
\end{figure}

\subsection{Explicit Heegaard diagrams for Seifert surfaces in $\boldsymbol{S^3}$} We now describe an explicit diagram adapted to a Seifert surface in $S^3$.

To begin,  we isotope $\Seif$ so that it consists of a disc with $2g$ bands  attached to it.   One way to do this is to note that since $\Seif$ is a  surface with one boundary
component, it is homotopy equivalent to a CW complex with one $0$-cell and $2g$
$1$-cells.  Represent the $0$-cell by a disk neighborhood, $D$, of an interior point.  Represent the $1$-cells by $2g$  disjoint arcs  on $\Seif$,
each  of whose  endpoints  lie in  $\partial  D$. Now let  $F$  be  a  regular
neighborhood  of $D$  and  the  $2g$ arcs.  Then  $\Seif$  deformation
retracts onto $F$ and $\partial \Seif$ stays an embedded
circle  in  $S$ throughout.  This  induces  an isotopy  which  takes  $(\Seif,K)$  to
$(F,\partial F)$. Indeed, if the surface $\Seif$ is presented in some nice
fashion then this gives  an algorithm  to  get such  a pair  $(F,\partial
F)$.   An  example   of   this  construction   is   shown  in   Figure
\ref{figureone} for the minimal genus Seifert surface of the trefoil.


Assuming, then, that $R$ is represented as above, proceed by contracting each band of $\Seif$ to an arc.  We may assume that the resulting disk with $2g$ arcs lies in a subset homeomorphic to $\R^3$ and, moreover, that there exists a plane $\R^2\subset \R^3$ onto which projection yields a planar diagram satisfying:
\begin{itemize}
\item The disc ($0$-handle) of $\Seif$ is embedded in $\R^2$ and no arc is projected to its interior.
\item The arcs have only finitely many transverse double points.
\end{itemize}

Keeping track of the crossing information at the double points, together with the framing of the band corresponding to each arc\footnote{Each band, $b$,  comes with a framing which is $\# \{ \text{full righ-handed twists\ of \ b}\} - \# \{ \text{full left-handed twists\ of \ b}\}.$},  we obtain a planar diagram from which we can recover the  surface, up to isotopy.  Let $k$ be the number of double points (crossings) in this planar diagram.

 Proceed by thickening  the  diagram  in   $\mathbb{R}^2$  to  obtain  a  handlebody, $H_\beta$,  in $\mathbb{R}^3$.  The genus of $H_\beta$ is $2g+k$, and we  let  $\Sigma=\partial H_\beta$  denote its  boundary.  Intersecting  $\Sigma$  with   the  original  plane results in $(2g+k+1)$ circles. Choose any $(2g+k)$ of these to be the $\alpha$  circles.  Clearly, $(\Sigma,\alpha_1,\ldots,\alpha_{2g+k})$ represents the complement of $H_\beta$.

Corresponding to each of the $2g$ arcs we obtain a $\beta$ circle. This circle is the boundary of a disk which intersects the arc in a point close to the disc, $D$.   Label these circles, $\beta_i$, $i=1,\dots,2g$. Finally, to each crossing in  the planar diagram we add a $\beta$ circle according to the convention of Figure \ref{fig:crossing}.  These circles are labeled $\beta_{2g+1},\dots\beta_{2g+k}$.

The resulting diagram $$(\Sigma, \{\alpha_1,\ldots,\alpha_{2g+k}\},\{\beta_{1},\ldots,\beta_{2g+k}\})$$ represents $S^3$ and contains $\Seif$ by construction.  An example of  such a  Heegaard diagram  for the case  of the  trefoil is shown  in Figure  \ref{figureone}.  

Next, we remove $8g$ discs from the Heegaard diagram while simultaneously adding $4g-1$ pairs of $\alpha$ and $\beta$ circles to adapt the diagram to $\Seif$.

To describe this, first observe that the  Heegaard surface  can be divided into two  parts; the {\em handles} and the {\em (punctured) sphere}.  The handles arise from the
boundary  of a  regular  neighborhood  of the  arcs  in the  planar
diagram,  while the sphere comes from the boundary  of  a regular
neighborhood  of the  disc $D.$    Next, note that $K=\partial \Seif$ is naturally embedded in $\Sigma$ in such a way that it does not intersect the circles $\beta_{2g+1},\ldots,\beta_{2g+k}$, and
intersects each  of the  other $\beta$ circles  exactly twice. See Figure \ref{fig:trefoilwithknot}

Up to isotopy in $\Sigma$, we can assume that  $K$ consists of two parallel strands in each  handle. We can further assume that $K$
lies mostly in the top part of $\Sigma$ (namely, the
part of $\Sigma$ lying above  the plane where the diagram of $\Seif$ was  embedded).  It passes to the bottom of $\Sigma$ only in the handles to account for the  framing and crossing information of the bands.  For each of the $4g$ intersection points $\{{p_{2i-1},p_{2i}}\}\in K\cap \beta_i$, $i=1,\dots,2g$, we remove two small discs from $\Sigma$ next to $p$, one on either side of $\beta_i$.  We denote the new surface-with-boundary by $\Sigma'$.    Removing the discs separates $K$ into $8g$ arcs, $\{A_j,B_j\}$, $j=1,\dots,4g$, with  $p_j\in B_j$.  For each $j\ne 4g$, we add a curve,  $\alpha_{2g+k+j}$, which encircles $B_j$ together with the boundary components of $\Sigma$ created by removing the discs near the endpoints of $B_j$.  Similarly, for each $j\ne 4g$ we add a  curve, $\beta_{2g+k+j}$, which encircles  $A_j$ and the boundary components of $\Sigma$ created by removing the discs near the endpoints of $A_j$.

Since $\Sigma$ contained $\Seif$ before the discs were removed, it is straightforward to construct a quasi-polygon, $P\subset \Sigma'$, which satisfies the requirements of Definition \ref{defn:16}.  Indeed, the $A$ and $B$ edges of $P$ are isotopic to the $A$ and $B$ arcs of the previous paragraph, as in Figure \ref{fig:modify}.  We have arrived at a surface diagram for $\Seif$:
$$(\Sigma', \{\alpha_1,\ldots,\alpha_{6g+k-1}\},\{\beta_{1},\ldots,\beta_{6g+k-1}\},P)$$

\noindent See Figure \ref{figtwo} for the diagram adapted to the Seifert surface of the trefoil.

\begin{figure}[htb]
\begin{center}
\includegraphics[width=475pt]{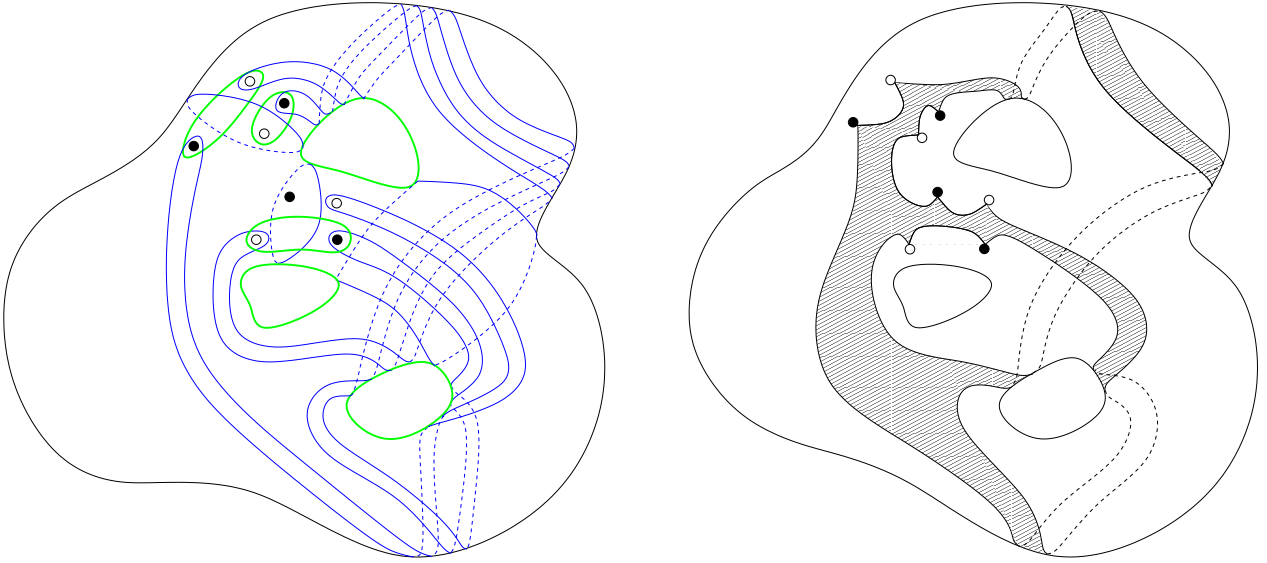}
\end{center}
\caption[section]{The Heegaard diagram}\label{figtwo}
\end{figure}

\begin{figure}[htb]
\begin{center}
\includegraphics[width=375pt]{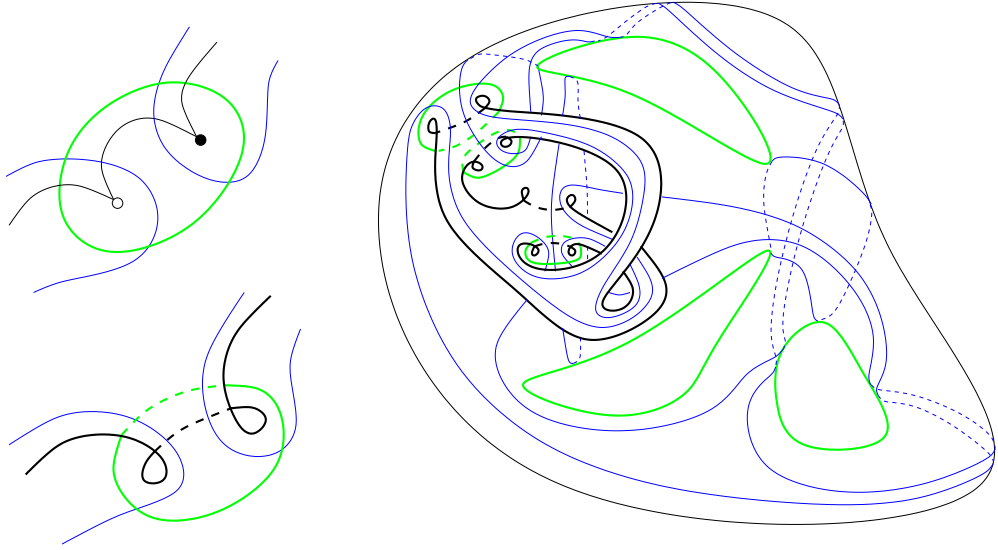}
\end{center}
\caption{The sutured diagram}\label{figthree}
\end{figure}

It is now easy to obtain a balanced diagram for the sutured manifold complementary to $\Seif$:  Delete the interior  of  $P$ from  the Heegaard  diagram and  take two copies of  the subsurface. Delete  all the $\alpha$ arcs  in one of the copies  and identify its $B$-arcs with  the corresponding $B$-arcs
in the Heegaard  diagram. Similarly, delete all the $\beta$  arcs in the other
copy, and identify all its $A$-arcs with the corresponding $A$-arcs in
the Heegaard diagram.  The process is shown locally  in the first part
of Figure  \ref{figthree}. The resulting balanced diagram represents $S^3(\Seif)$.   The Heegaard surface
has  genus $(7g+k-1)$ and $1$  boundary component,  and has
$(6g+k-1)$  $\alpha$  circles  and  $\beta$ circles  each.  The  final
sutured  diagram   for  the  trefoil   example  is  shown   in  Figure
\ref{figthree}.


\begin{figure}

\psfrag{R}{$\Seif$}
\psfrag{t}{$\Seif\times I$}
\psfrag{i}{Isotopy}
\psfrag{b}{$\beta$}
\psfrag{a}{Add $1$-handle}
\begin{center}
\includegraphics[width=450pt]{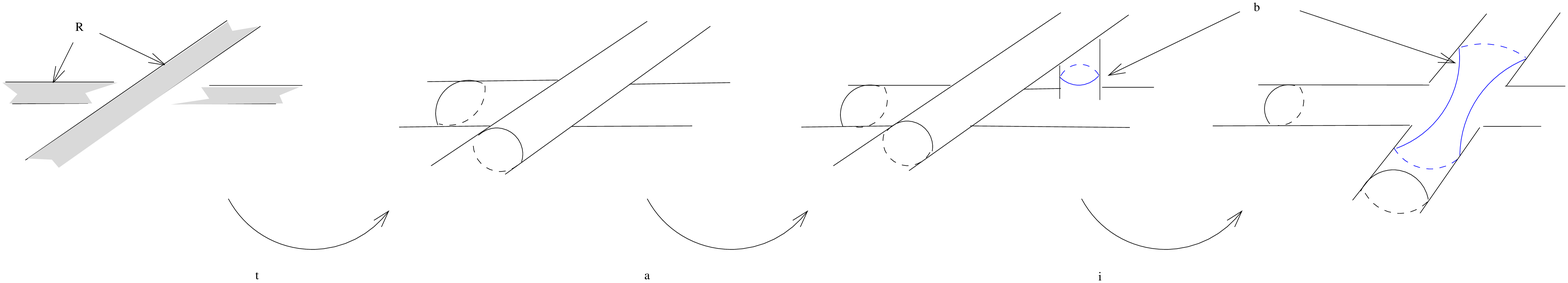}
\caption{\label{fig:crossing2}}
\end{center}
\end{figure}

\begin{rem} The diagram above is a special case of the general construction discussed in the previous two subsections.    The procedure by which we handled the crossing regions of the planar diagram associated to $\Seif$  is equivalent to adding $1$-handles to $\Seif \times I$ to make the complement into a handlebody.  See Figure \ref{fig:crossing2}.  Removing discs and adding $\alpha/\beta$ pairs is simply a specific implementation of the algorithm from Subsection \ref{subsec:algorithm}.   Indeed, the diagram of this subsection can be seen as extremal: at every step in the algorithm we  used a stabilization.  The other extremal case, where we use only isotopies, will be implemented in Section \ref{sec:example} to calculate the sutured Floer homology for the Seifert surfaces of $8_3$.
\end{rem}
\begin{rem}
 The diagram above is usually not the minimal genus possible. Picking a different set of $1$-handles to add to $\Seif\times I$ will frequently lower the genus of the Heegaard diagram significantly.  The minimal number of $1$-handles necessary to add to an embedded handlebody, $H_g$, so that the complement is a handlebody is often referred to as the {\em tunnel number } of $H_g$.
\end{rem}

While we do not use this diagram for the computation in the next section,  it may be of future use to note that the combinatorics of the diagram enable one to calculate $SFH(S^3(\Seif))$ without decomposing the diagram.  More precisely, recall from  the discussion surrounding Theorem \ref{thm:decomposition} that the generators of the chain complex associated to the decomposed diagram are in bijection with the outer generators $C(\mathcal{O}_P)$ on the surface diagram.   Indeed, Theorem \ref{thm:decomposition} shows that the outer generators form a subcomplex of $SFH(S^3_{2n}(K))$ whose homology is isomorphic to $SFH(S^3(\Seif))$.   However,  it is not clear that the obvious bijection between generators of $C(\mathcal{O}_P)$ and $SFH(S^3(\Seif))$ induces the isomorphism on homology.  The differentials on these two complexes could be quite different, as their definition is in terms of quite different Heegaard diagrams.  It is only after altering the surface diagram severely to make it ``nice" that an identification between the differentials is established.  In light of this, it is nice to know that we can compute the homology of $SFH(S^3(\Seif))$, as a relatively $H_1(S^3\setminus \Seif)$ graded group, without decomposing the surface diagram.  Indeed we have the following proposition

\begin{prop} Let $\mathcal{H}=(\Sigma,\boldsymbol{\alpha},\boldsymbol{\beta},P)$ be the explicit surface diagram for a Seifert surface $\Seif\subset S^3$ described above, and let $\mathcal{H}'=(\Sigma',\boldsymbol{\alpha}',\boldsymbol{\beta}')$ be the decomposed diagram.  Denote the associated chain complexes  by $(C(\mathcal{H}),\partial)$ and $(C(\mathcal{H}'),\partial')$. Then the differential on the subcomplex $C(\mathcal{O}_P)\subset C(\mathcal{H})$ generated by outer intersection points is equal to that on $C(\mathcal{H}')$, under the obvious isomorphism of chain groups induced by the bijection of generators.  In particular, the relative $\SpinC$ grading on $SFH(S^3(\Seif))$ can be computed by considering the difference $\epsilon(\x,\y)$ of two generators $\x,\y\in C(\mathcal{O}_P)$ as a $1$-cycle in $H_1(S^3\setminus \Seif)$ \end{prop}

\begin{rem} For more details on how to regard the difference $\epsilon(\x,\y)$ of outer generators as an element in  $H_1(S^3\setminus \Seif)$, see Subsection \ref{subsubsec:relativegr} below.
\end{rem}

\begin{proof}
Let $\x,\y \in C(\mathcal{O}_P)\subset C(\mathcal{H})$ denote outer generators; that is, generators whose intersection points lie outside the quasi-polygon, $P$, representing $\Seif$.  Let $\x',\y'\in C(\mathcal{H}')$ denote the corresponding generators for the decomposed diagram.  It suffices to show that \begin{equation}\label{eq:partial}
\y\in \partial \x  \text{ if and only if } \y'\in \partial'\x'
\end{equation} (recall that we are working with $\Z/2\Z$ coefficients).  Let $\phi\in \pi_2(\x,\y)$ be a domain connecting the outer generators.  Proving  \eqref{eq:partial} will be accomplished by showing that  $\widehat{\Mod}(\phi)\ne \emptyset$ implies $\phi\cap \{B\text{-arcs}\}=\emptyset$.  In other words,  the domains which contribute to $\partial$ do not intersect the $B$-arcs on the quasipolygon.  To see why this implies \eqref{eq:partial}, note that any $\phi$ satisfying $\phi\cap \{B\text{-arcs}\}=\emptyset$ can  be thought of as $\phi' \in \pi_2(\x',\y'),$ actually such domains are in one to one correspondence with domains on the decomposed diagram that support holomorphic representatives. Indeed, if $\phi'$ is a domain connecting $\x'$ and $\y'$ that has a holomorphic representative, then projecting this holomorphic map to $\Sigma$ we see that $\phi = p(\phi')$ also has a holomorphic representative, thus $\phi \cap \{B\text{-arcs}\} = \emptyset.$ 
Furthermore, for such domains, an almost complex structure on $\Sigma\times D^2$ achieving transversality for ${\Mod}(\phi)$ can be extended to an almost complex structure (under the embedding of $\Sigma \setminus B \subset \Sigma'$ away from the sutures)  on $\Sigma'\times D^2$ which achieves transversality for  ${\Mod}(\phi')$ (here we are thinking in terms of the cylindrical version of Floer homology \cite{Lipshitz}). In this way we see that if $\phi\cap \{B\text{-arcs}\}=\emptyset$, then $\widehat{\Mod}(\phi)\ne \emptyset$ if and only if $\widehat{\Mod}(\phi')\ne \emptyset.$ Examining each $\phi$, we obtain \eqref{eq:partial}. 

Thus we must only show that $\widehat{\Mod}(\phi)\ne \emptyset$ implies $\phi\cap \{B\text{-arcs}\}=\emptyset$. By contradiction, suppose that $\widehat{\Mod}(\phi)\ne \emptyset$ and $\phi\cap \{B\text{-arcs}\} \neq \emptyset$ for some domain $\phi.$
In  Figure  \ref{figfour},  we   have  marked  various  components  of
$\mathcal{H}\setminus(\bm{\alpha}\cup\bm{\beta})$.      The     quasipolygon
$P$  is  shaded.  As  usual,  the thick  green  circles  are
$\alpha$ circles, and  the thin blue circles are  $\beta$ circles. Let
$n_p$ be the  multiplicity  of $\phi$ in a  region marked $p$.   Recall from Lemma \ref{lem:positivity} that if $\widehat{\Mod}(\phi)\ne \emptyset$, then $n_p\ge 0$ for all $p\in \Sigma$.

Our first observation deals with  the parts of the diagram that appear locally  like the lower left part of the  figure.  We claim that for  $\x,\y\in \mathcal{O}_P$, all domains  satisfy  $n_k-n_i=n_l-n_j$.  This follows from the fact that $\x,\y\in \mathcal{O}_P$ implies there are no points of $\x,\y\in P$, and hence there are no corner points of $\phi$ contained in $P$.

Next we deal with the arcs in $\phi$ which which hit the $B$ arc shown  in the first  part of
that figure.  Label the  regions in the handle part of  $\mathcal{H}$ by $a,b,c$ and $g$  and let the regions in the  sphere part of $\mathcal{H}$  be $d,e,f$
and $h$. There are two
cases.

\begin{itemize}
\item $n_d\neq 0$. Since $n_e=0$ ($e$ contains a suture) and $\x$ and $\y$ contain no
points in $P$,  $n_h-n_f=n_d-n_e>0$. Thus $n_h>0$, and since
$h$ is in the spherical part of $\mathcal{H}$, the region marked $h$ is  the same region as
the  disk region  of  $\Seif$, marked  $m$  in the  third  part of  Figure
$\ref{figfour}$. But the disk part  contains a suture, leading to a
contradiction.

\item $n_c\neq 0$.  Again, since $n_b=0$, a  similar argument shows that
$n_g-n_a>0$.  Proceed by examining the multiplicities of $\phi$ in the  regions of the handle  part of $\mathcal{H}$ which border the $\beta$
curve that separated  $a$ from $g$. Of course there may be  $\alpha$ curves encountered
on the way, as shown in the lower left part of Figure $\ref{figfour}$.  However, our
first  observation above shows that the difference of the multiplicities
of $\phi$  on the two regions adjacent to  the $\beta$ curve
stays positive. Proceeding along the handle until we reach the disk region, we  again get $n_m>0$. Thus in either case, we are done.

\end{itemize}

\end{proof}

\begin{figure}[htb]
\psfrag{A}{a}
\psfrag{B}{b}
\psfrag{C}{c}
\psfrag{D}{d}
\psfrag{E}{e}
\psfrag{F}{f}
\psfrag{G}{g}
\psfrag{H}{h}
\psfrag{I}{i}
\psfrag{J}{j}
\psfrag{K}{k}
\psfrag{L}{l}
\psfrag{M}{m}
\psfrag{b}{B}
\psfrag{a}{A}

\begin{center}\includegraphics[width=350pt]{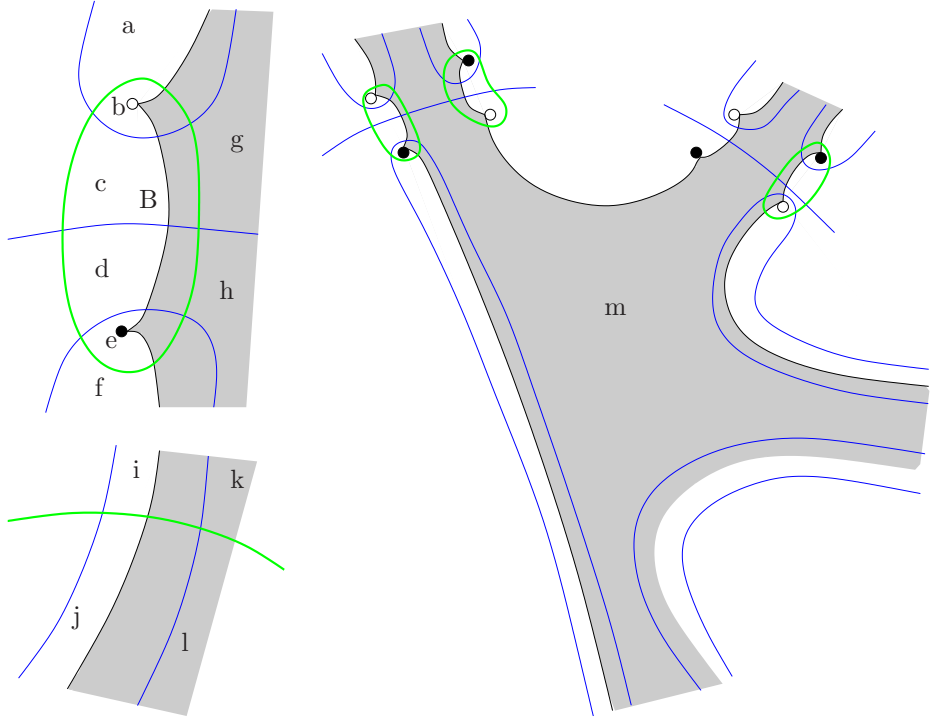}
\end{center}
\caption{Local      coefficients     of     $\phi$      at     various
points}\label{figfour}
\end{figure}

\section{Using $SFH(Y(R))$ to distinguish Seifert surfaces}
\label{sec:example}
Given an oriented knot $K\subset S^3$, there are several notions of equivalence one could consider for its Seifert surfaces.    We will consider two Seifert surfaces, $\Seif_1,\Seif_2$, to be {\em equivalent} if there is an isotopy of $S^3$ taking $\Seif_1$ to $\Seif_2$.  Note that this  is the same as considering $\Seif_1$ and $\Seif_2$ to be equivalent if there is an orientation preserving diffeomorphism between the pairs $(S^3,\Seif_1)$ and $(S^3,\Seif_2)$ (since the group of orientation preserving diffeomorphisms of the three-sphere is path-connected).    A more restrictive notion, called {\em strong equivalence},  regards $\Seif_1$ and $\Seif_2$ as equivalent if they are isotopic in the complement of $K$.
Note that we can discuss whether the surfaces $R_1$ and $R_2$ are equivalent if $\partial R_1$ and $\partial R_2$ are equivalent knots, while we can say that two surfaces are strongly equivalent only if $\partial\Seif_1 = \partial\Seif_2.$

The remainder of this section will be devoted to showing how the sutured Floer homology invariants of $S^3(\Seif)$ can be used to distinguish non-equivalent Seifert surfaces.  We do this through a detailed discussion of an example.  Figure \ref{fig:knot} depicts two Seifert surfaces, $\Seif_1$ and $\Seif_2$, each bounded by the knot $8_3$. Note that $\Seif_1$ is obtained by plumbing two bands with $+2$ and $-2$ full twists, respectively, and $\Seif_2$ is obtained by taking the dual plumbing. So indeed the two surfaces bound the same oriented knot. We will show that $\Seif_1$ and $\Seif_2$ are inequivalent.  Combining this with the results of \cite{Kakimizu}, it will follow that these represent all isotopy classes of Seifert surface for $8_3$ (see Proposition \ref{prop:all} below).

\begin{figure}
\psfrag{c1}{$c_2$}
\psfrag{c2}{$c_1$}
\psfrag{d1}{$d_2$}
\psfrag{d2}{$d_1$}
\psfrag{S}{$R_1$}
\psfrag{T}{$R_2$}

\begin{center}
\includegraphics[width=350pt]{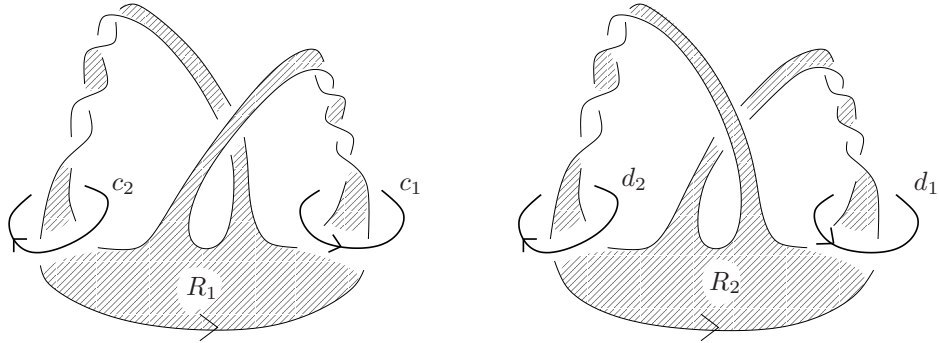}
\end{center}
\caption{\label{fig:knot}Two different Seifert surfaces for the same knot, $8_3$. The outward normal to each surface on the shaded region is out of the plane of the page (towards the reader).}
\end{figure}

\subsection{Classical Methods}\label{subsec:classical}
Before beginning, we make some preliminary remarks regarding this particular example and the applicability of previously known techniques.  As mentioned in the introduction, one effective way to distinguish isotopy classes of surfaces is through the fundamental group of their complements, see \cite{Alford}.  In the present case, however, this obviously fails.  Indeed,  $S^3\setminus\Seif_1$ is homeomorphic to $S^3\setminus\Seif_2$; they are both open handlebodies of genus $2$.  Thus any attempt to use the fundamental group will be fruitless.

Even in the case that the complements of the surfaces are homeomorphic,  classical techniques could still be of use.   The Seifert form provides a useful obstruction to finding an isotopy between two surfaces.  To describe this, recall that $H_1(R;\Z)$  is equipped with a bilinear form, $$Q_R: H_1(R;\Z)\otimes H_1(R;\Z)\rightarrow \Z,$$ called the Seifert form. Given
curves $a$ and $b$  in $\Seif$, let $b^+$ be a push-off  of $b$ in the direction specified by the positive unit normal vector field of $\Seif$. Then the  Seifert form evaluated   on $([a],[b])$ is the  linking  number of $a$  with  $b^+$  in  $S^3$.   Suppose now that two Seifert surfaces are isotopic.  It follows that they have congruent Seifert forms.  This means that there exists $W\in SL(2g,\Z)$ for which $$ V_{R_2}=W^T V_{R_1} W,$$
where $V_{R_i}$ are integral matrices representing $Q_{R_i}$ with respect to given bases for $H_1(R_i;\Z)\cong \Z^{2g}$.  Concretely, $W$ is the matrix representing the isomorphism of $H_1(R;\Z)$ induced by the diffeomorphism $(S^3,R_1)\cong (S^3,R_2)$.  Thus, to show that two Seifert surfaces are inequivalent, it suffices to show that their Seifert forms are not congruent (see \cite{Trotter} for applications of this method to Seifert surfaces of some pretzel knots).    In the situation at hand, however, this method also fails.  For the obvious symplectic bases, the following matrices represent the Seifert forms of $R_1$ and $R_2$:

\[ V_{R_1}= \left(\begin{array}{cc}
 2&0\\ 1&-2
\end{array} \right)\ \ \ \ \ \ \ \ \ \ \ \ \ \ \   V_{R_2}= \left( \begin{array}{ccc}
 2&-1\\ 0&-2
\end{array} \right). \] The intersection forms in the same basis are represented by
$$ U_{R_1}=  U_{R_2}=\left(\begin{array}{cc}
 0&1\\ -1&0
\end{array} \right).$$
One can easily check, however, that $V_{R_1}$ and $V_{R_2}$ are congruent.  An appropriate element of  $SL(2,\Z)$ is
\[W=\left( \begin{array}{cc}
 4&-5\\ -3& 4
\end{array} \right). \]
Note that $W$ also preserves the standard symplectic form on $\mathbb{Z}^2,$ i.e., $U_{R_2} = W^TU_{R_1}W,$ so the Seifert form and the intersection form together are incapable of distinguishing $R_1$ and $R_2.$

Finally, we remark that techniques from the theory of sutured manifolds have been quite fruitful in studying Seifert surfaces up to strong equivalence. See, for instance, \cite{Kob, Kakimizu}.  Indeed \cite{Kakimizu} has used these techniques to classify minimal genus Seifert surfaces up to strong equivalence for knots of $10$ or fewer crossings.   In particular, it follows that if $R_1$ and $R_2$ are dual plumbings of a $+2$ and a $-2$ twisted band, then they represent distinct strong equivalence classes, and that these are the only two such classes. It should be noted, however, that equivalence and strong equivalence are quite different.  For instance, if the bands both had framing $+2,$ then the boundary of the resulting dual surfaces would each be the knot $7_4$.  It follows from  \cite{Kakimizu}  that these two surfaces are strongly inequivalent.  It is easy to verify, however, that they are isotopic, and hence equivalent in our sense.  As our techniques are able to distinguish surfaces up to isotopy, we will make no further reference to strong equivalence.

\subsection{The Technique}\label{subsec:technique}
Suppose that two surfaces $R_1$ and $R_2$ are isotopic.  It follows that the complementary sutured manifolds, $S^3(R_1)$ and $S^3(R_2)$, will be equivalent.  Thus to show that $R_1$ and $R_2$ are inequivalent, it suffices to show that the sutured Floer homology groups $SFH(S^3(R_1))$, $SFH(S^3(R_2))$ are different.

The algorithm from the previous section tells us how to construct Heegaard diagrams adapted to the surfaces.   From these diagrams, we can identify generators for the chain complexes and determine the difference between the relative $\SpinC$ structures associated to generators, $\x$ and $\y$.

After analyzing the chain groups, we will determine their homology indirectly  through consideration of the Euler characteristic.   The total rank of the groups will agree for $R_1$ and $R_2$, since it equals the rank of the top group of the knot Floer homology of $\partial R_i=8_3$. Thus the heart of the argument is to distinguish the groups by showing that their $\RelSpinC$ gradings are different.

Given $\spinc_1,\spinc_2\in \RelSpinC(S^3(R_i))$ which support non-trivial Floer groups, our analysis of the chain complexes will produce a geometric representative for the difference class  PD$[\spinc_1-\spinc_2] \in H_1(S^3\setminus R_i;\Z)$.    To show that the Floer homology groups of $S^3(R_1)$ and $S^3(R_2)$ are different, we thus need a way to distinguish the various difference classes in $H_1(S^3\setminus R_1;\Z)$ from those in $H_1(S^3\setminus R_2;\Z)$.  This is rather subtle, however, since the presentation for $H_1(S^3\setminus R_i;\Z)$ depends on the Heegaard diagram (or, equivalently, the presentation of $R_i$), and there could be an equivalence between $S^3(R_1)$ and $S^3(R_2)$ which induces an automorphism of $H_1(S^3\setminus \Seif;\Z)$.

To remove this ambiguity we use the Seifert form of the surface.  To describe this, let $\Seif$  be a  Seifert surface  for a knot $K\subset S^3$, and let $M=S^3\setminus \mathrm{Int} (\Seif\times I)$ be the complement of a regular
neighborhood of  $\Seif$.   Then we  have the
following natural isomorphisms.

\begin{align*}
H_1(\Seif)&\cong H_1(\Seif,\partial \Seif)&\text{(Long    exact   sequence  for  the pair)}\\
&\cong H^1(\Seif)&\text{(Poincar\'e Duality)}\\
&\cong H^2(S^3,\Seif)&\text{(Long  exact  sequence  for  the pair)}\\
&\cong H^2(M,\partial M)&\text{(Excision)}\\
&\cong H_1(M). &\text{(Poincar\'e Duality)}
\end{align*}

Since the Seifert form is invariant under isotopy of $\Seif$,  the above
isomorphisms  endow $H_1(M)\cong H_1(S^3\setminus \Seif)$ with a bilinear form which we also denote by $Q_{\Seif}$. Given $a,b\in H_1(S^3\setminus \Seif;\Z)$, let us denote $Q_{\Seif}(a,b)$ by $a\cdot b$. Similarly, using the above isomorphisms, we can endow $H_1(S^3 \setminus R)$ with another bilinear form which is obtained from the intersection pairing on $H_1(R).$ Its value on the pair $(a,b)$ will be denoted by $a \cap b.$  The discussion shows that if $h \colon (S^3,R_1) \to (S^3,R_2)$ is an orientation preserving homeomorphism then $h_*\colon H_1(S^3 \setminus R_1) \to H_1(S^3 \setminus R_2)$ satisfies $a \cdot b = h_*(a) \cdot h_*(b)$ and $a \cap b = h_*(a) \cap h_*(b).$

Let $\langle\,c_1,c_2\,\rangle$ and $\langle\,d_1,d_2\,\rangle$ be bases of $H_1(S^3 \setminus \Seif_1)$ and $H_1(S^3 \setminus \Seif_2),$ respectively, as shown on Figure \ref{fig:knot}. Tracing through the isomorphisms in our particular examples shows that matrix representations for $Q_{\Seif_i}$ are also given by the matrices $V_{R_i}$ above.   Thus $c_i\cdot c_j$ (resp. $d_i\cdot d_j$) is given by the $ij$-th entry of $V_{R_i}$. The values of $a\cdot b$ will distinguish the difference classes discussed above which, in turn, will distinguish the sutured Floer homology as relatively graded groups.

\subsection{Calculation}

Consider the two Seifert surfaces $\Seif_1,\Seif_2$ for $8_3$ shown in Figure \ref{fig:knot}.  In this section, we calculate the sutured Floer homology groups of $S^3(R_i)$, showing that $SFH(S^3(R_1))\ncong SFH(S^3(R_2))$.  We discuss $R_1$ in detail, and then summarize the results for $R_2$.

\begin{figure}
\psfrag{a1}{$\alpha_1$}
\psfrag{a2}{$\alpha_2$}
\psfrag{a3}{$\alpha_3$}
\psfrag{b1}{$\beta_1$}
\psfrag{b2}{$\beta_2$}
\psfrag{b3}{$\beta_3$}
\psfrag{A}{A}
\psfrag{B}{B}
\psfrag{X}{x}
\begin{center}
\includegraphics[width=350pt]{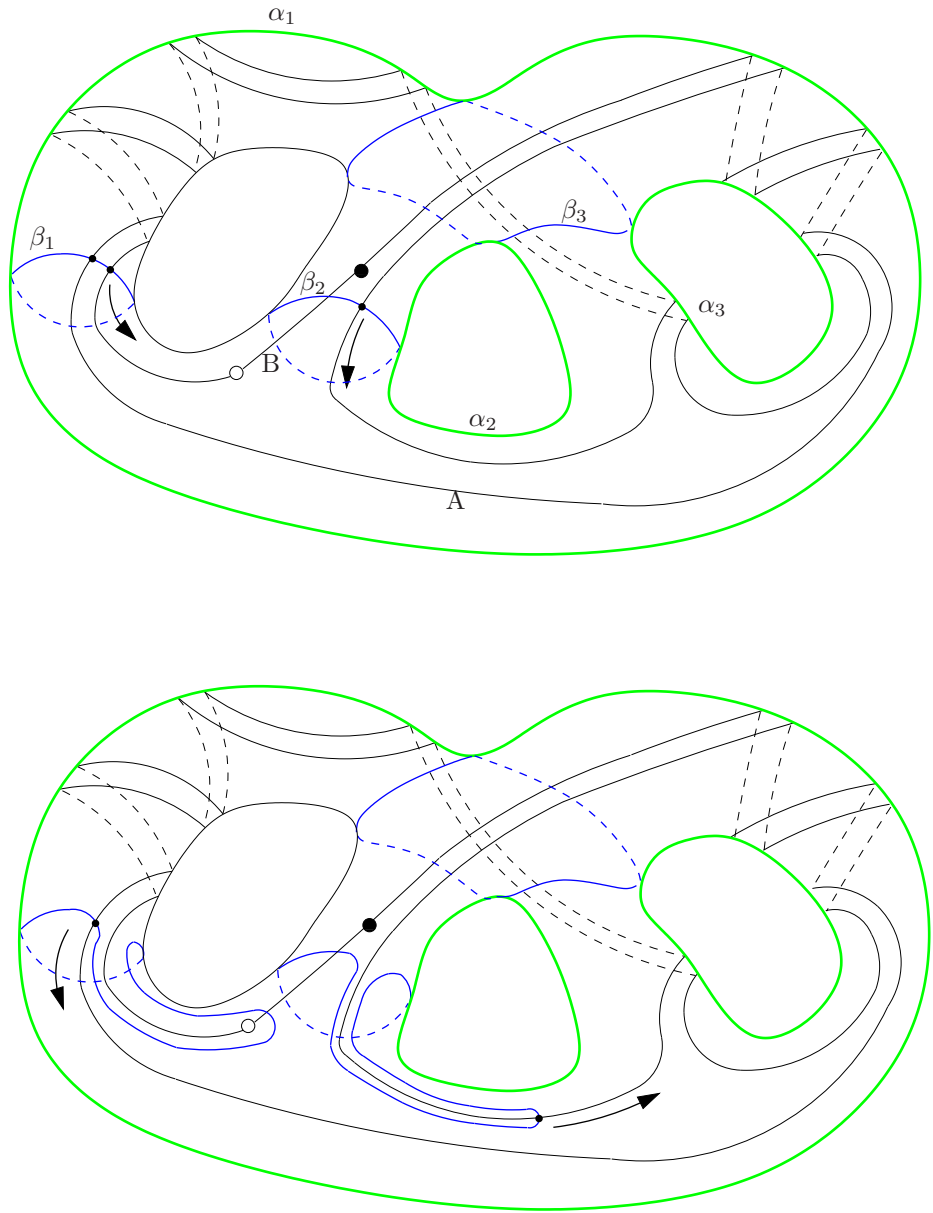}
\caption{\label{fig:makediagram1}Constructing the surface diagram for $\Seif_1$. The first figure shows a Heegaard diagram for $S^3$ containing $\Seif_1$ as a proper subsurface.  The thick green curves are $\alpha_i$, $i=1,2,3$ while the thin blue curves are $\beta_i$, $i=1,2,3$.  There are three intersection points between the $A$ arc and the $\beta$ curves, which we remove by a sequence of isotopies.}
\end{center}
\end{figure}

\begin{figure}
\psfrag{A}{A}
\psfrag{B}{B}
\psfrag{P}{P}
\begin{center}
\includegraphics[width=350pt]{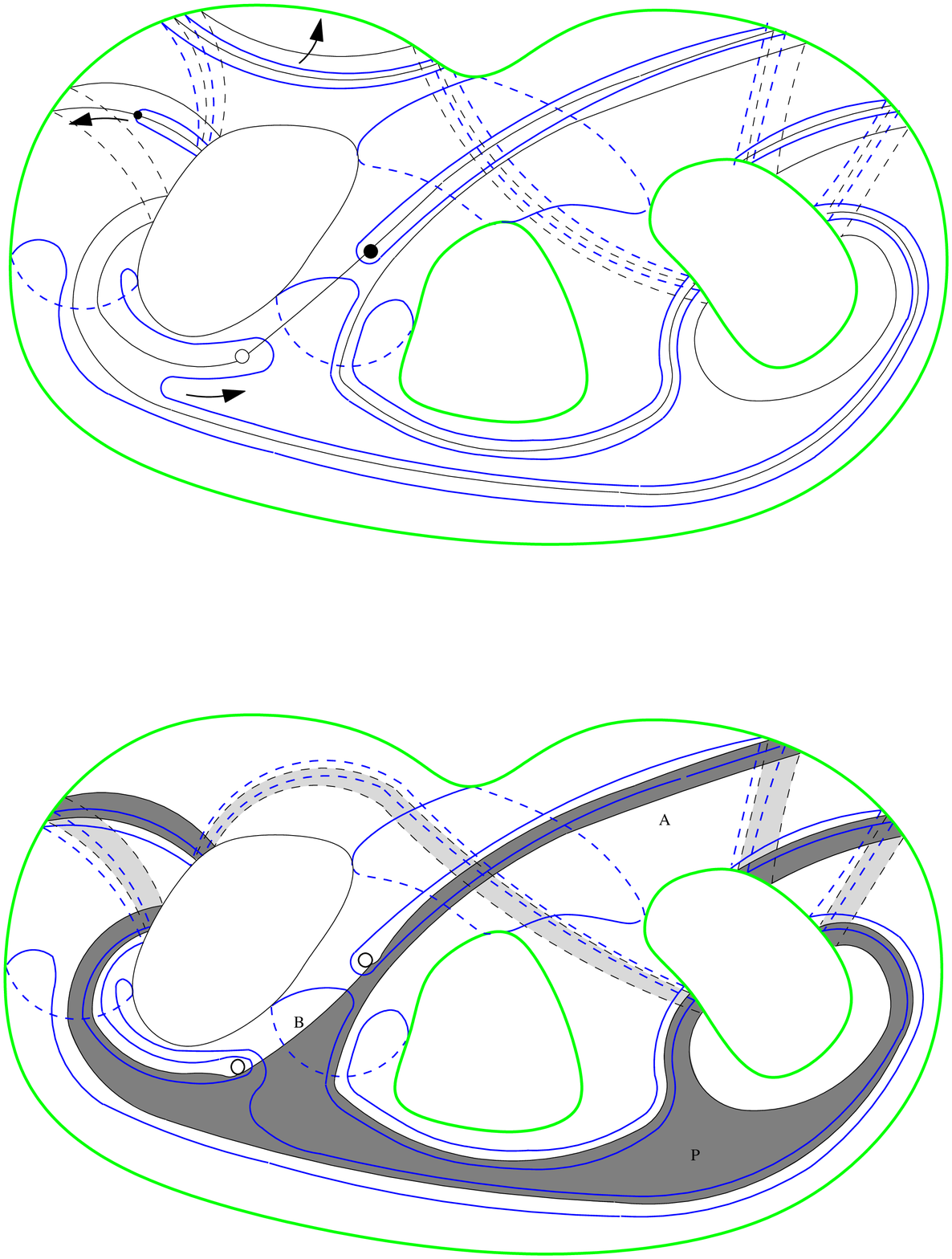}
\caption{\label{fig:makediagram3} The bottom figure is the surface diagram for $\Seif_1$.  We have shaded the quasi-polygon which represents $\Seif_1\subset S^3 \setminus K$. }
\end{center}
\end{figure}

\subsubsection{Drawing the diagram}
Figures \ref{fig:makediagram1} and \ref{fig:makediagram3} illustrate the construction of the surface diagram for the first surface, $\Seif_1$.   Our choice of basepoints ensures that the resulting $B$ arc of $K=\partial \Seif_1$ does not intersect the $\alpha$ curves.   When removing the $3$ intersections of the $A$ arc with the $\beta$ curves, we use only  isotopies of the $\beta$ curves.  Thus the resulting surface diagram is extremal in the sense that we do not use any stabilizations in the algorithm from Subsection \ref{subsec:algorithm}.

The bottom of Figure \ref{fig:makediagram3} shows the surface diagram, i.e., the  Heegaard diagram for the sutured manifold $S^3_2(K)$ (the knot complement with $2$ parallel meridional sutures) with $\Seif_1$ appearing as a quasi-polygon.

At this point, it is straightforward to construct the sutured Heegaard diagram for $S^3(\Seif_1)$.  Simply remove the quasi-polygon representing $\Seif_1$ from $\Sigma$ (the shaded region labeled $P$ in Figure \ref{fig:makediagram3}), and glue two copies of it to what remains in such a way that the gluing is along $A$ arcs on one copy and $B$ arcs on the other.    This is shown in Figure \ref{fig:R1decomposed}

After the decomposition, there are no intersection points of $\alpha$ and $\beta$ curves lying on the two quasipolygons, $P_A,P_B$  which we glued to $\Sigma \setminus P$.    Thus, the $3$-tuples of intersection points which comprise generators for the chain complex  will be contained in  $\Sigma\setminus P$.   For that reason,  it is convenient to erase all the $\beta$ arcs which intersect $P$ in the surface diagram of Figure \ref{fig:makediagram3}.  The resulting diagram is shown in Figure \ref{fig:R1}.  This latter diagram is simpler to work with, in general, and contains the homotopy theoretic data necessary to understand the chain complex as a relatively $H_1(S^3\setminus \Seif_1)$ graded group.



\subsubsection{The generators and their relative gradings}\label{subsubsec:relativegr}

From Figure \ref{fig:R1}, we see that there are $10$ generators for the sutured Floer chain complex.  We can label these generators by triples, where  $\x=x_1 x_2 x_3$ denotes the generator which
contains the point labeled $x_i$ on $\alpha_i$. 
Given generators $\x$ and $\y$, we wish to calculate the difference of their associated $\SpinC$ structures, $\spinc(\x)-\spinc(\y)$.     To do this, join each $x_i$ to $y_i$ by an oriented arc along $\alpha_i$, and then join $y_i$ to some $x_j$ by an oriented arc along a $\beta$ curve.  The result is a collection of closed curves, \diff \, whose homology class we denote by $\epsilon(\x,\y)=[\gamma_{\x,\y}]\in H_1(S^3 \setminus N(\Seif_1);\Z) \cong H_1(S^3\setminus \Seif_1;\Z)$, and refer to it as the {\em difference class} associated to $\x,\y$. According to Lemma \ref{lem:spincdiff}, the class $\epsilon(\x,\y)$ is Poincar{\'e} dual to $\spinc(\x)-\spinc(\y)\in H^2(S^3\setminus N(\Seif_1),\partial;\Z).$ Note that since the $\beta$ arcs  in Figure \ref{fig:R1} are connected,  \diff \ can be taken to lie entirely in that figure. A cycle representative for the difference class of $\x=512$ and $\y=313$ is shown in Figure \ref{fig:connect}.

  Ultimately, we wish to evaluate the Seifert form on the difference classes.  To this end, it will be convenient to express $\epsilon(\x,\y)$  in terms of the basis for $H_1(S^3\setminus \Seif_1;\Z)$  given by $c_1,c_2$ in Figure \ref{fig:knot}.   In order to do this,  push the
parts  \diff \ which lie on the $\alpha$ arcs towards the $\bm{\alpha}$
handlebody (i.e., outwards).  Similarly,  push the parts of \diff \ lying on
the $\beta$ curves towards the $\bm{\beta}$ handlebody
(i.e., inwards).  The result is a closed curve \difft \ which punctures the Heegaard surface only at the intersection points comprising $\bm{x}$ and $\bm{y}$.  See Figure \ref{fig:gamma2}.  Note, however, that the Heegaard surface contains the presentation of the Seifert surface in Figure \ref{fig:knot}.  Furthermore, \difft \ is in the complement of this presentation since the intersection points comprising $\bm{x}$ and $\bm{y}$ do not intersect  the Seifert surface.

\begin{figure}
\psfrag{a1}{$\alpha_1$}
\psfrag{a2}{$\alpha_2$}
\psfrag{a3}{$\alpha_3$}
\psfrag{b1}{$\beta_1$}
\psfrag{b2}{$\beta_2$}
\psfrag{b3}{$\beta_3$}
\psfrag{1}{$1$}
\psfrag{2}{$2$}
\psfrag{3}{$3$}
\psfrag{4}{$4$}
\psfrag{5}{$5$}
\psfrag{6}{$6$}
\psfrag{7}{$7$}
\begin{center}
\includegraphics[width=350pt]{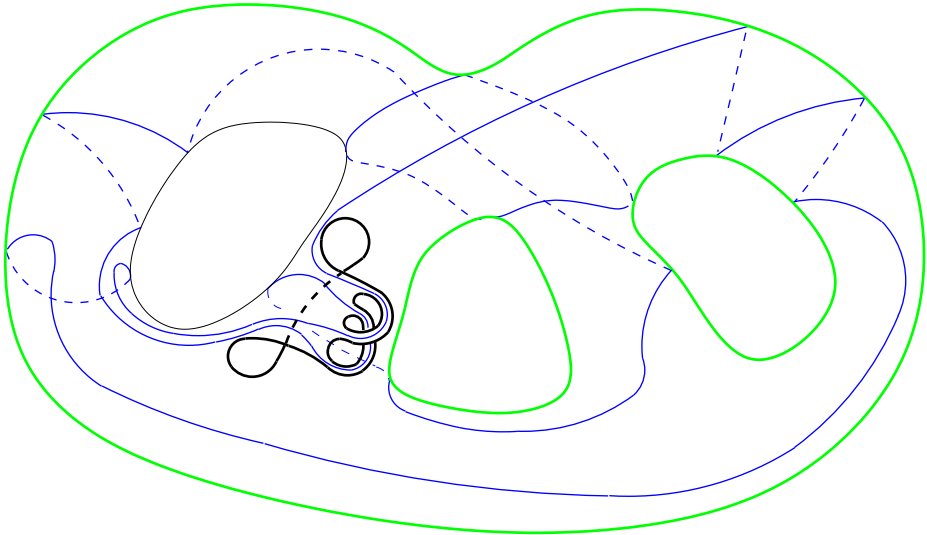}
\caption{\label{fig:R1decomposed}The decomposed diagram, i.e., the sutured diagram for $S^3(\Seif_1)$. The bold curve is the boundary of the Heegaard surface.}
\end{center}
\end{figure}

\begin{figure}
\psfrag{a1}{$\alpha_1$}
\psfrag{a2}{$\alpha_2$}
\psfrag{a3}{$\alpha_3$}
\psfrag{b1}{$\beta_1$}
\psfrag{b2}{$\beta_2$}
\psfrag{b3}{$\beta_3$}
\psfrag{1}{$1$}
\psfrag{2}{$2$}
\psfrag{3}{$3$}
\psfrag{4}{$4$}
\psfrag{5}{$5$}
\psfrag{6}{$6$}
\psfrag{7}{$7$}
\begin{center}
\includegraphics[width=350pt]{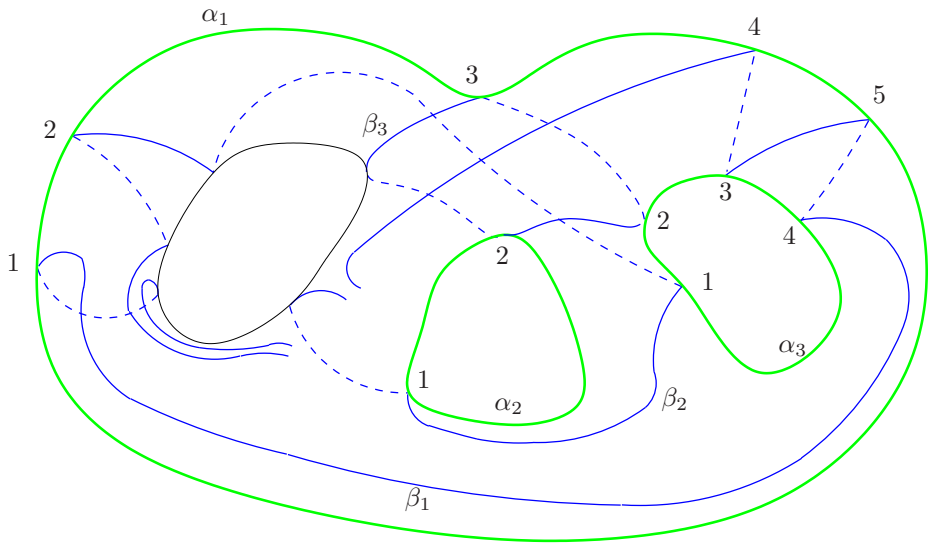}
\caption{\label{fig:R1} The surface diagram from Figure \ref{fig:makediagram3}, where we have erased all arcs of intersection $\beta\cap P$ of the $\beta$ curves with the quasipolygon $P$ representing $\Seif_1$.   This diagram contains all $3$-tuples of intersection points $\alpha_i\cap \beta_{\sigma(i)}$ which comprise the generators  of $SFH(S^3(\Seif_1))$.}
\end{center}
\end{figure}

\begin{figure}
\psfrag{a1}{$\alpha_1$}
\psfrag{a2}{$\alpha_2$}
\psfrag{a3}{$\alpha_3$}
\psfrag{b1}{$\beta_1$}
\psfrag{b2}{$\beta_2$}
\psfrag{b3}{$\beta_3$}
\psfrag{1}{$1$}
\psfrag{2}{$2$}
\psfrag{3}{$3$}
\psfrag{4}{$4$}
\psfrag{5}{$5$}
\psfrag{6}{$6$}
\psfrag{7}{$7$}
\psfrag{g}{\diff}
\begin{center}
\includegraphics[width=350pt]{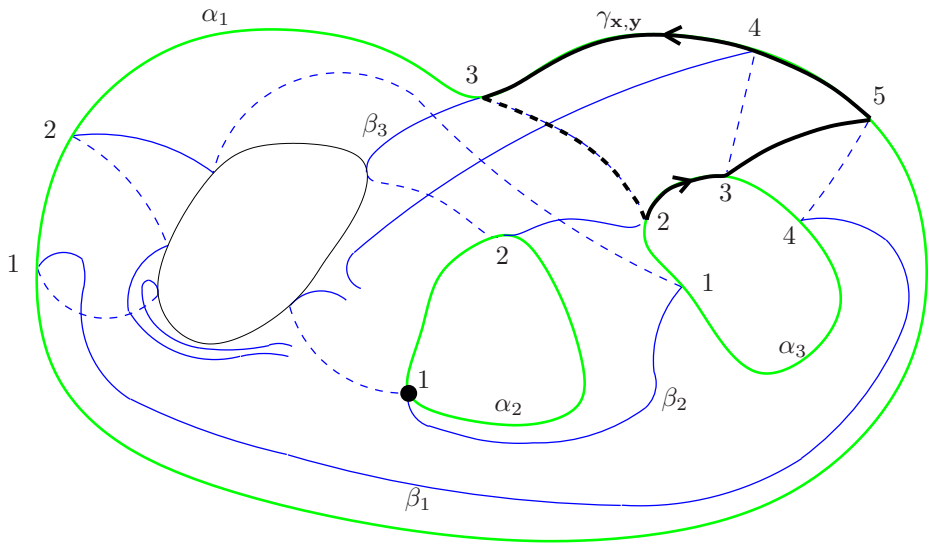}
\caption{\label{fig:connect} Construction of a cycle representative, $\gamma$, of the difference class, $\epsilon(\x,\y)=\mathrm{PD}[\spinc(\x)-\spinc(\y)]\in H_1(S^3\setminus \Seif_1)$.  Here $\x=512$, $\y=313$.  The representative is comprised of two curves, one of which is constant at $1\in \alpha_2\cap \beta_2$.}
\end{center}
\end{figure}

\begin{figure}
\psfrag{g}{\difft}
\psfrag{c}{$c_1$}
\psfrag{eq}{$\simeq$}

\begin{center}
\includegraphics[width=350pt]{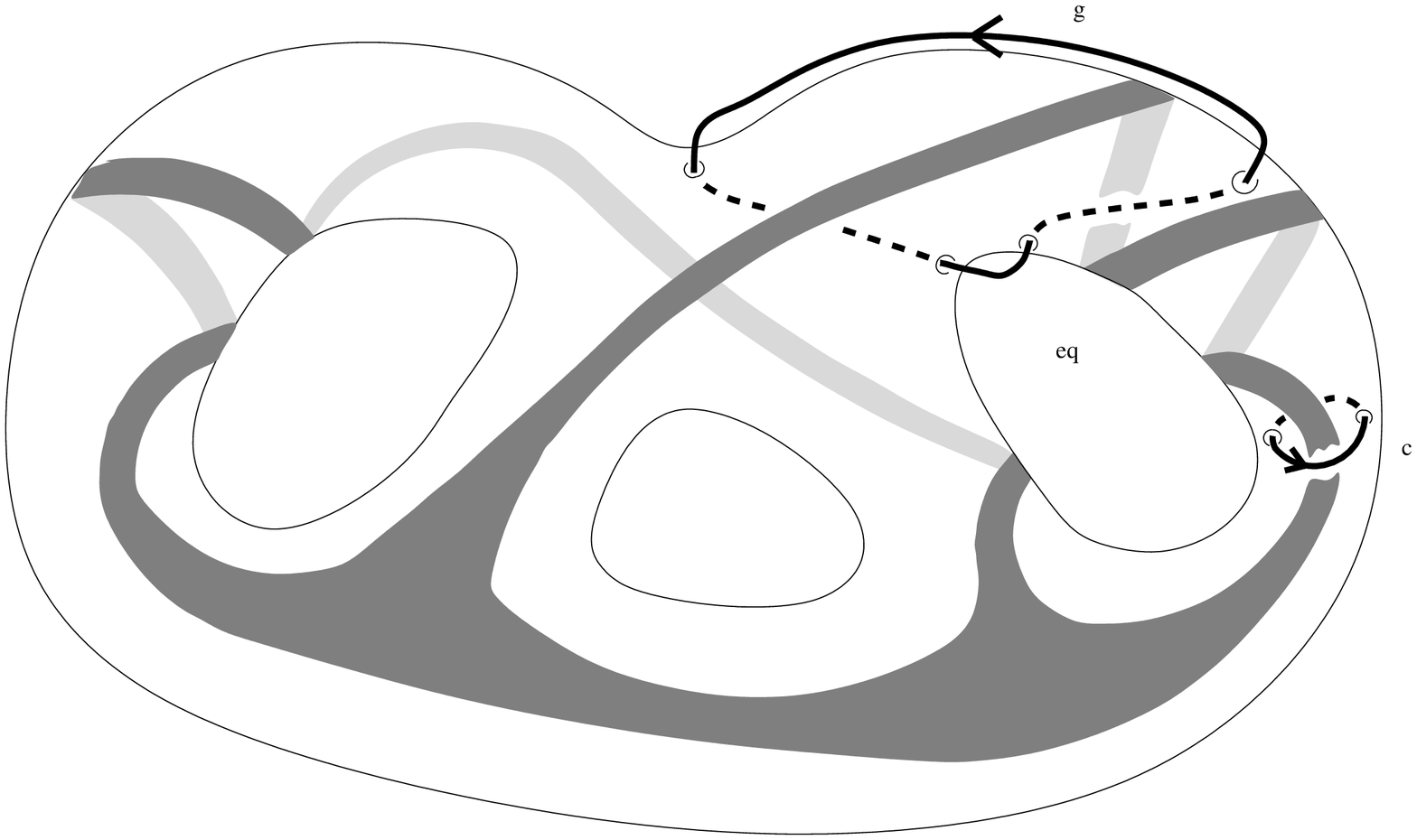}
\caption{\label{fig:gamma2} The push-off of \diff \ into the handlebodies yields a curve \difft \ living in the complement of $\Seif_1$.   In this example, \difft \ is homologous to $c_1$.}
\end{center}
\end{figure}

Thus we can regard \difft \ in two ways: as a curve in the sutured manifold $S^3(\Seif_1)$ presented by the decomposed diagram or as a curve in the complement of $\Seif_1$, as shown in Figure \ref{fig:knot}.  We claim that the homology class which \difft \ represents in $H_1(S^3\setminus \Seif_1)$ is the same, regardless of which way we view it.  In this way, we can determine the difference classes in terms of the basis for $H_1(S^3\setminus \Seif_1)$ given by $c_1,c_2$ in Figure \ref{fig:knot}.

To prove the claim, one need only trace through the construction of the sutured diagram, starting from the presentation of $\Seif_1$ in Figure \ref{fig:knot}.  We began by constructing a diagram for $S^3$ which contained $\Seif_1$ as a subsurface (the top part of Figure \ref{fig:makediagram1}).  Regarding \difft \ on this diagram, its homology class clearly agrees with that obtained by thinking of it as a curve in Figure \ref{fig:knot}.  Indeed, removing $\beta_1$ and $\beta_2$ from this diagram specifies  $S^3\setminus \Seif_1$ as it is presented in Figure \ref{fig:knot}  (see Remark  \ref{rem:surfacecomp}).   It follows that isotopies of $\beta_1,\beta_2$ do not change the homology class of \difft \  in $S^3\setminus \Seif_1$.  The surface diagram differs from the diagram for $S^3$ containing $\Seif_1$ only by a sequence of isotopies, followed by the removal of two disks to turn it into a sutured diagram for $S^3_2(K)$.  This latter modification, however, is  performed far
  from \difft \  and hence does not effect its homology class.   Another way to see this is that the surface diagram, by definition, specifies an embedding of the Seifert surface in the knot complement  $S^3\setminus K$.    From its construction, this embedding differs from the embedding shown in Figure \ref{fig:knot} only by an  isotopy  supported in an arbitrarily small neighborhood of $\partial \Seif_1$.  Thus, the homology class of \difft \ in $H_1(S^3\setminus \Seif_1)$ specified by the surface diagram agrees with that of Figure \ref{fig:knot}.   Finally,  decomposing the diagram is exactly the same as decomposing $S^3_2(K)$.  Since \difft \ is in the complement of the quasipolygon, its homology class in $H_1(S^3\setminus \Seif_1)$ (as specified by the Heegaard diagram)  is unchanged by the decomposition.  This proves the claim.

 Figures \ref{fig:connect} and \ref{fig:gamma2} indicate that the difference class between $\x=512$ and $\y=313$ is  $c_2$.  The remaining differences are easily computed,  and the following diagram represents the chain complex as a relatively $H_1(S^3\setminus \Seif_1)$ graded group.    (The explanation for the diagram is that each $\x$ is placed on a lattice point in the affine lattice generated by $c_1,c_2$. The difference between the lattice coordinates of $\x$ and $\y$ is the difference class $\epsilon(\x,\y)$.)

\xymatrix{& & & & & & \ar@{-}[d]^{c_2}   {223}    \ar@{-}[r]^-{c_1}   &{224} \ar@{-}[d]^{c_2} \\ \
& &  & & & & {313,412,421}  \ar@{-}[r]^{c_1}   &{314,512,521} \ar@{-}[r]^-{c_1}  &{112,121}
}



\subsubsection{The homology}
 We take our chain complexes with $\Z/2\Z$ coefficients. Our first observation is that the rank of the homology of the chain complex above  is $4$.  This follows from Theorem \ref{thm:topterm} above, together with the fact that $\mathrm{rk} \ \HFKa(K,1) =4.$
This latter fact can be seen in many ways and follows, for instance from the fact that  $8_3$ is an alternating knot of genus $1$ for which the top coefficient of the Alexander polynomial equals  $-4$. (According to Theorem $1.3$ of \cite{OSz9}, for alternating knots, one has$$\mathrm{rk}\  \HFKa(K,i) = |a_i|,$$ where $a_i$ is the $i$-th coefficient  of the symmetrized Alexander polynomial $\Delta_K(T)=a_0 +\sum_i a_i(T^i + T^{-i})$.)  Now the homology of each of the $2$ subcomplexes in the top row of the diagram is $\Z/2\Z$; indeed, each complex has a single generator.   This takes care of a $2$-dimensional subspace of the $4$-dimensional homology.  As for the rest of the homology, note that the subcomplex generated by $112,121$ must have even Euler characteristic while the subcomplexes generated by $314,512,521$ and   $313,412,421$, respectively have odd (in particular, non-zero) Euler characteristics.   The only way for this to happen is if the $112,121$ subcomplex is
 acyclic and the homology of the remaining two subcomplexes is $\Z/2\Z$.  Summarizing, the sutured Floer homology as a relatively $H_1(S^3\setminus \Seif_1)$ graded group is

\xymatrix{  & & & & & & \ar@{-}[d]^{c_2}   {\Z/2\Z}    \ar@{-}[r]^-{c_1}   &{\Z/2\Z} \ar@{-}[d]^{c_2} \\ \
& &  & & & & {\Z/2\Z}  \ar@{-}[r]^-{c_1}   &{\Z/2\Z}
}

\subsubsection{The results for $\Seif_2$}

\begin{figure}
\psfrag{a1}{$\alpha_1$}
\psfrag{a2}{$\alpha_2$}
\psfrag{a3}{$\alpha_3$}
\psfrag{b1}{$\beta_1$}
\psfrag{b2}{$\beta_2$}
\psfrag{b3}{$\beta_3$}
\psfrag{1}{$1$}
\psfrag{2}{$2$}
\psfrag{3}{$3$}
\psfrag{4}{$4$}
\psfrag{5}{$5$}
\psfrag{6}{$6$}
\psfrag{7}{$7$}
\psfrag{g}{\diff}
\begin{center}
\includegraphics[width=350pt]{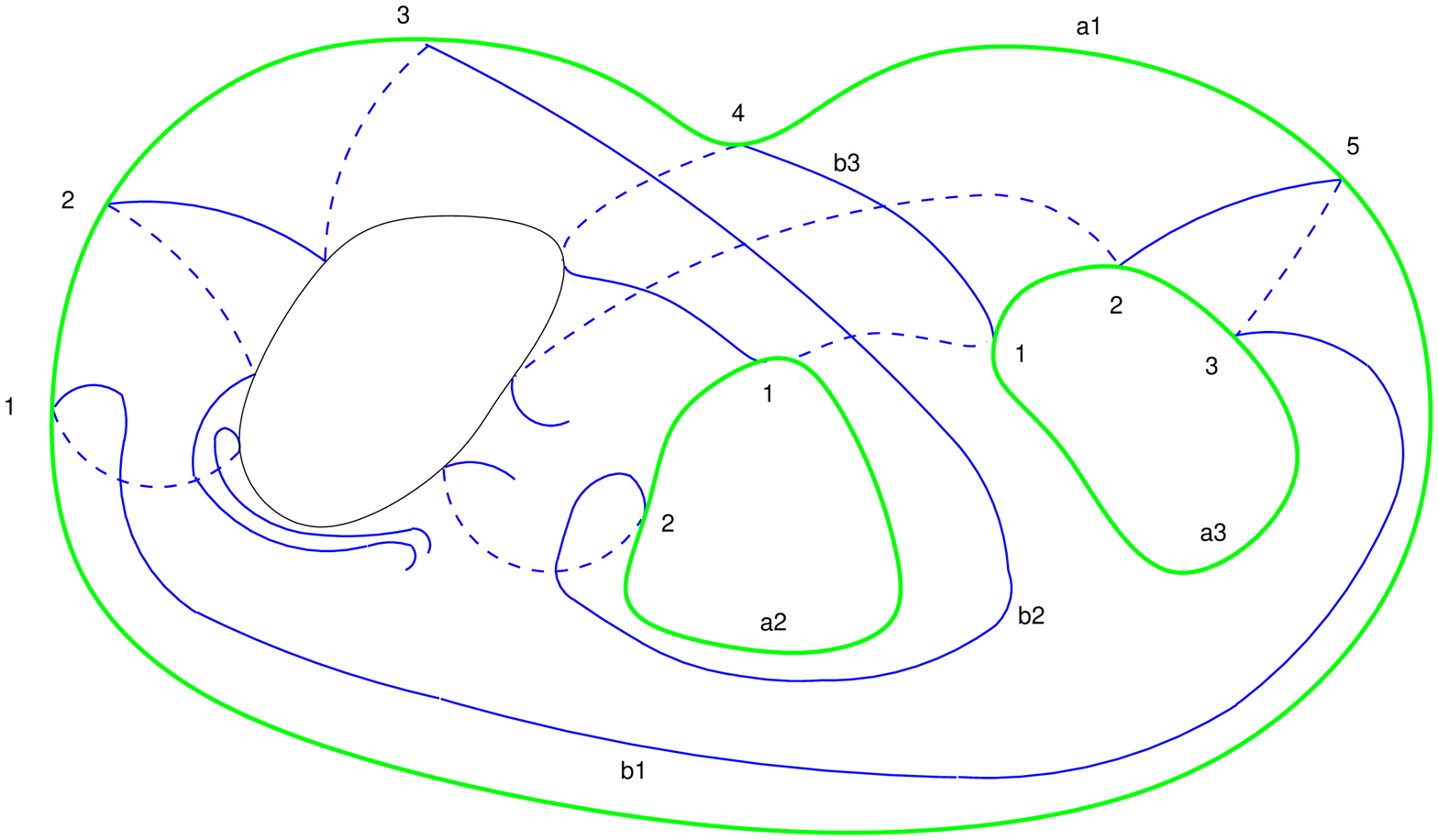}
\caption{\label{fig:R2} Diagram for $\Seif_2$.}
\end{center}
\end{figure}

Figure \ref{fig:R2} shows the Heegaard diagram for $\Seif_2$.  The resulting chain complex has $8$ generators, whose relative gradings are given in terms of the basis $d_1,d_2$ of Figure \ref{fig:knot} as follows \xymatrix{& & & & & & {212}  \ar@{-}[r]^{d_1}  \ar@{-}[d]^{d_2} &{213}  \ar@{-}[d]^{d_2} \\ \ & & & & & &  {312}  \ar@{-}[r]^-{d_1} \ar@{-}[d]^{d_2} &{313}  \ar@{-}[d]^{d_2}
\ \\ & & & & & &  {422,521}  \ar@{-}[r]^-{d_1}  &{423,121}
  }

As above, the rank of the sutured Floer homology for $S^3(\Seif_2)$ is $4$.  Since the Euler characteristic of the subcomplexes generated by $422,521$ and $423,121$ are even, the sutured Floer homology is given by

\xymatrix{\,\,\,\,\,\, & & & & & & \ar@{-}[d]^{d_2}   {\Z/2\Z}    \ar@{-}[r]^-{d_1}   &{\Z/2\Z} \ar@{-}[d]^{d_2} \\ \
\,\,\,\,\,\, & &  & & & & {\Z/2\Z}  \ar@{-}[r]^-{d_1}   &{\Z/2\Z}
}

\subsubsection{Distinguishing the groups}
Having calculated the sutured Floer homology of $S^3(\Seif_1)$ and $S^3(\Seif_2)$ as relatively graded groups, it remains to distinguish them.

We first remind the reader what it means for two collections of groups graded by relative $\SpinC$  structures to be isomorphic.

\begin{defn}\label{defn:SFHiso} Two relative $\SpinC$-graded groups
$$ SFH(M_1,\gamma_1)=\!\!\! \!\! \underset{  \spinc\in \RelSpinC(M_1,\gamma_1) }\bigoplus \!\!\! \!\!\!\!\!\!SFH(M_1,\gamma_1, \spinc)\ \ \ \ , \ \ \ \ SFH(M_2,\gamma_2)=\!\!\! \!\! \underset{ \spinc\in \RelSpinC(M_2,\gamma_2) }\bigoplus \!\!\! \!\!\!\!\!\!SFH(M_2,\gamma_2, \spinc)$$
are {\em isomorphic} (which we denote $SFH(M_1,\gamma_1)\cong SFH(M_2,\gamma_2))$  if
\begin{enumerate}
\item There is an isomorphism, $f^*:H^2(M_1,\partial M_1;\Z)\rightarrow H^2(M_2,\partial M_2;\Z).$
\item There is a bijection of sets $\sigma: \RelSpinC(M_1,\gamma
_1)\rightarrow \RelSpinC(M_2,\gamma
_2).$
\item The following diagram commutes

\xymatrix{   & & \ar  @{} [dr]  |  {}
 \RelSpinC(M_1,\gamma_1)\times H^2(M_1,\partial M_1;\Z) \ar[d] \ar[r]^{(\sigma,f^*)} & \RelSpinC(M_2,\gamma_2)\times H^2(M_2,\partial M_2;\Z)\ar[d] \\
& & \RelSpinC(M_1,\gamma_1) \ar[r]^{\sigma} & \RelSpinC(M_2,\gamma_2)}
\noindent where the vertical arrows are induced by the action of $H^2(M_i,\gamma_i)$ on $\RelSpinC(M_i,\gamma_i)$.
\item  There are isomorphisms $g_\spinc: SFH(M_1,\gamma_1, \spinc)\rightarrow SFH(M_2,\gamma_2, \sigma(\spinc))$ for every $\spinc\in \RelSpinC(M_1,\gamma_1).$

\end{enumerate}

 \end{defn}
If $f: (M_2,\gamma_2)\rightarrow (M_1,\gamma_1)$ is an equivalence, then Theorem \ref{thm:invariance} indicates that $SFH(M_1,\gamma_1)\cong SFH(M_2,\gamma_2).$    In this case,  $f^*$ and $\sigma$ in Definition \ref{defn:SFHiso} are both obtained by pull-back along  $f$.  In addition, if $f$ comes from the restriction of an equivalence of Seifert surfaces $(S^3,\Seif_2) \to (S^3,\Seif_1),$   then  $f_*: H_1(S^3 \setminus \Seif_1)\rightarrow H_1(S^3 \setminus \Seif_2)$ preserves  the Seifert form discussed in Subsection \ref{subsec:technique}, i.e.,  $a\cdot b = f_*(a)\cdot f_*(b)$.

Returning to our example, this can be made  concrete as follows.  Let us denote generators for the $4$ non-zero Floer homology groups of $R_1$ (resp. $R_2$)  by  $\x_i$  (resp. $\y_i$), so that the Floer homology groups are given by

\xymatrix{ & & & \ar@{-}[d]^{c_2}   {<\x_2>}    \ar@{-}[r]^-{c_1}   &{<\x_1>} \ar@{-}[d]^{c_2}  & & & \ar@{-}[d]^{d_2}   {<\y_2>}    \ar@{-}[r]^-{d_1}   &{<\y_1>} \ar@{-}[d]^{d_2} \\ \
 & & & {<\x_4>}  \ar@{-}[r]^-{c_1}   &{<\x_3>} & &
 & {<\y_4>}  \ar@{-}[r]^-{d_1}   &{<\y_3>}
}
\noindent where, $< - >$ means the $\Z/2\Z$ vector space generated by $-$.  Then, in order for $SFH(S^3(\Seif_1))$ to be isomorphic to  $SFH(S^3(\Seif_2))$, there must be a bijection between sets
$$\sigma: \{\x_1,\x_2,\x_3,\x_4\} \longrightarrow \{\y_1,\y_2,\y_3,\y_4\},$$
which is compatible with taking difference classes, i.e.,
$\epsilon(\sigma(\x_i),\sigma(\x_j))=f_*\epsilon(\x_i,\x_j)$, for some isomorphism $f_*:  H_1(S^3\setminus \Seif_1)\rightarrow  H_1(S^3\setminus \Seif_2)$.  Since we assume that $\Seif_1$ is equivalent to $\Seif_2$, $f_*$ must preserve the Seifert form.

Suppose that $\sigma$ exists.  Then we have
$$ \epsilon(\sigma(\x_1),\sigma(\x_2))^2=f_*\epsilon(\x_1,\x_2)^2=\epsilon(\x_1,\x_2)^2=c_1^2=2,$$
where squares indicate the pairing, under the Seifert form, of a class with itself.  Thus the difference of $\sigma(\x_1)$ and $\sigma(\x_2)$ is a class whose square is $2$.  Considering $\epsilon(\y_i,\y_j)$ for every $i\ne j$, we obtain $8$ distinct classes $$\begin{array}{cccccccccc}
\pm d_1 &  & & \pm d_2 & & &  \pm (d_1+d_2) & & & \pm (d_1-d_2),
\end{array}$$

\noindent whose squares are $$\begin{array}{ccccccccccccccccccc}
2 &  & & &     -2 & & & & &  &  -1 & & & & & &  1, & &
\end{array}$$\noindent respectively.   This shows that $\epsilon(\sigma(\x_1),\sigma(\x_2))=\pm d_1$.   Similarly, the fact that $$ \epsilon(\sigma(\x_1),\sigma(\x_3))^2=f_*\epsilon(\x_1,\x_3)^2=\epsilon(\x_1,\x_3)^2=c_2^2=-2,$$
\noindent implies that $\epsilon(\sigma(\x_1),\sigma(\x_2))=\pm d_2$.
We have arrived at a contradiction. For on the one hand $$\epsilon(\sigma(\x_1),\sigma(\x_2))\cdot \epsilon(\sigma(\x_1),\sigma(\x_3)) = \epsilon(\x_1,\x_2)\cdot  \epsilon(\x_1,\x_3)= c_1 \cdot c_2 = 0,$$ while on the other we have shown
$$\epsilon(\sigma(\x_1),\sigma(\x_2))\cdot \epsilon(\sigma(\x_1),\sigma(\x_3))= \pm d_1 \cdot \pm d_2,$$
\noindent and this latter pairing is non-zero, regardless of signs.  This shows that $SFH(S^3(\Seif_1))\ncong SFH(S^3(\Seif_2))$, and hence $\Seif_1\not\simeq \Seif_2$.

\begin{rem}  Our argument shows that there does not exist an orientation-preserving diffeomorphism of $S^3$ which takes $\Seif_1$ to $\Seif_2$.  There is, however, an obvious orientation-reversing diffeomorphism which sends  $\Seif_1$ to $\Seif_2$.  To see this, simply reflect $\Seif_1$ across the plane of the page, then rotate $180^\circ$ around a vertical axis through the middle of the surface.  The composition of the reflection and rotation is the aforementioned diffeomorphism.  Our result, then, can be interpreted as saying that sutured Floer homology detects ``chirality" of Seifert surfaces. On the other hand, it is an interesting fact that the knot $8_3$ is fully amphichiral.
\end{rem}

\subsection{A few consequences}
\begin{prop} \label{prop:all}$\Seif_1$ and $\Seif_2$ represent all equivalence classes of minimal genus Seifert surfaces for $8_3$.
\end{prop}
\begin{proof}
By \cite{Kakimizu} the knot $8_3$ has exactly two Seifert surfaces up to strong equivalence, namely $\Seif_1$ and $\Seif_2.$ We have
just seen above that these two surfaces are inequivalent. The result follows.
\end{proof}

\begin{thm}For any $n\ge1$, there exists a knot $K_n$ with Seifert surfaces $\{F_0,\dots F_n\}$, such that $F_i$ is not equivalent to $F_j$ for any $i\ne j$.
\end{thm}
\begin{proof}
Take $K_n$ to the connected sum of $n$ copies of $8_3$, and let $F_i$ be the Seifert surface obtained by forming the boundary connected sum of  $i$ copies of $R_1$ and $n-i$ and copies of the $R_2$.  Note that we can perform $n-1$ product disk decompositions\footnote{A product disk decomposition is a surface decomposition along a properly embedded disk which intersects the sutures in exactly $2$ points.} to  $S^3(F_i)$ to obtain a sutured manifold which is equivalent to the disjoint union of $i$ copies of $S^3(\Seif_1)$ and $n-i$ copies of $S^3(\Seif_2)$.  Now sutured Floer homology is unchanged under product decompositions (see Lemma $9.13$ of \cite{sutured}), and under disjoint union behaves according to the K{\"u}nneth principle:
$$SFH(Y_1\sqcup Y_2, \gamma_1\sqcup \gamma_2,\spinc_1 \sqcup \spinc_2) \cong SFH(Y_1,\gamma_1,\spinc_1)\otimes SFH(Y_2,\gamma_2,\spinc_2).$$
(where we work with $\Z/2\Z$ coefficients to avoid any Tor terms).  Using these facts together with the calculation from the previous section, we can distinguish the number of copies of $R_1$ used to form $F_i$ as follows.  First, observe that rk $SFH(S^3(F_i))=4^n$.  A generating set for the Floer homology of $F_i$ is given by $n$-tuples  ${\bf x_jy_m}$, with ${\bf j}=\{j_1,\dots,j_i\}, {\bf m}=\{m_1,\dots m_{n-i}\}$ and $j_l,m_k \in \{1,2,3,4\}$.  The $n$-tuple ${\bf x_jy_m}$ corresponds to
$$\x_{j_1}^1\otimes \dots \otimes \x_{j_i}^i \otimes \y_{m_{1}}^{1}\otimes   \dots \otimes \y_{m_{n-i}}^{n-i}, $$
 where $\x_{j_{l}}^l$ (resp. $\y_{j_{l}}^l$) is one of the $4$ generators of the sutured Floer homology of the $l$-th copy of $R_1$ (resp. $R_2$) used to form $F_i$. The difference classes associated to the generators are then given by
$$ \epsilon( {\bf x_jy_m} , {\bf x_{j'}y_{m'}}) = \epsilon(\x_{j_1},\x_{j'_1})\oplus \dots \oplus \epsilon(\x_{j_i},\x_{j'_i}) \oplus \epsilon(\y_{m_1},\x_{m'_1})\oplus \dots \oplus \epsilon(\y_{m_1},\y_{m'_1})
,$$
with $\epsilon(\x_{j_l},\x_{j'_l})$  (resp. $\epsilon(\y_{j_l},\y_{j'_l})$) one of the $8$ distinct differences   $\pm c_1^l, \pm c_2^l,\pm (c_1^l+c_2^l), \pm (c_1^l-c_2^l)$  (resp. $\pm d_1^l, \pm d_2^l,\pm (d_1^l+d_2^l), \pm (d_1^l-d_2^l).$ Here, as throughout, the upper indices on $c_k^l$ (resp. $d_k^l$) are used to  denote an element in $H_1(S^3\setminus F_i)$ which comes from the $l$-th copy of $R_1$ (resp. $R_2$). Finally, we note that the Seifert form on $H_1(S^3(F_i))\cong \Z^{2n}$ splits as a sum,
$$ Q_{F_i}=Q_{R_1^1} \oplus \dots\oplus Q_{R_1^i}\oplus Q_{R_2^1}  \oplus \dots\oplus Q_{R_2^{n-i}}.$$

Now suppose that $F_i$ is isotopic to $F_k$ for some $i\ne k$. As in the previous section, this implies that there is a bijection between generators
$$\sigma: \{ {\bf x_jy_m}\} \rightarrow  \{ {\bf \tilde{x}_j\tilde{y}_m}\},$$
which is compatible with an isomorphism $f_* : H_1(S^3\setminus F_i)\rightarrow  H_1(S^3\setminus F_k)$, which preserves the Seifert form.  We use $\sim $ to distinguish generators for $F_k$ from those  for $F_i$. Abusing notation, for $k \in \mathbb{Z}$ let ${\bf k}=\{k,k,\dots,k\}$ denote the vector of any length, all of whose entries are $k.$  We have
$$ \epsilon(\sigma({\bf x_1y_1}),\sigma({\bf x_{2}y_{2}}))^2=f_*\epsilon({\bf x_1y_1},{\bf x_{2}y_{2}})^2=\epsilon({\bf x_1y_1},{\bf x_{2}y_{2}})^2=(c_1^1+\dots+c_1^i+d_1^1+\dots d_1^{n-i})^2=2n.$$
It follows that $$\epsilon(\sigma({\bf x_1y_1}),\sigma({\bf x_{2}y_{2}}))= \delta_1 c_1^1 + \dots + \delta_k c_1^k + \rho_1 d_1^1 + \dots + \rho_{n-k} d_1^{n-k}, $$ for some choice of signs $\delta_l, \rho_l \in \{-1,1\}$.  Indeed these are the only elements in $H_1(S^3\setminus F_k)$ of square $2n$ which arise as differences of generators.  A similar analysis shows that $$\epsilon(\sigma({\bf x_1y_1}),\sigma({\bf x_{3}y_{3}}))= \delta_1' c_2^1 + \dots + \delta_k' c_2^k + \rho_1' d_2^1 + \dots + \rho_{n-k}' d_2^{n-k}, $$
as these are the only elements in $H_1(S^3\setminus F_k)$ of square $-2n$.
We claim that the signs must agree in both cases; that is, $\delta_l'=\delta_l$ and $\rho_l'=\rho_l$ for all $l$.  To see this observe that $f_*$, in addition to preserving the Seifert form, must preserve the intersection product on $H_1(S^3\setminus F_i)$ inherited from $H_1(F_i)$ (equivalently, $f_*$ must be a symplectomorphism of the symplectic vector space $H_1(S^3 \setminus F_i;\R)$).  Denoting this product by $\cap$, we have $$\epsilon({\bf x_1y_1},{\bf x_{2}y_{2}})\cap \epsilon({\bf x_1y_1},{\bf x_{3}y_{3}})= (c_1^1+\dots+c_1^i+d_1^1+\dots d_1^{n-i}) \cap (c_2^1+\dots+c_2^i+d_2^1+\dots d_2^{n-i}) = n.$$  On the other hand, $$\epsilon(\sigma({\bf x_1y_1}),\sigma({\bf x_{2}y_{2}}))\cap \epsilon(\sigma({\bf x_1y_1}),
\sigma({\bf x_{3}y_{3}}))= f_*\epsilon({\bf x_1y_1},{\bf x_{2}y_{2}})\cap f_*\epsilon({\bf x_1y_1},{\bf x_{3}y_{3}}) =$$
$$=(\delta_1 c_1^1 + \dots + \delta_k c_1^k + \rho_1 d_1^1 + \dots + \rho_{n-k} d_1^{n-k}) \cap (\delta_1' c_2^1 + \dots + \delta_k' c_2^k + \rho_1' d_2^1 + \dots + \rho_{n-k}' d_2^{n-k})= $$ $$ \delta_1\delta_1' c_1^1\cap c_2^1 + \dots + \delta_k\delta_k' c_1^k\cap c_2^k + \rho_1\rho_1' d_1^1\cap d_2^1 + \dots + \rho_{n-k}\rho_{n-k}' d_1^{n-k}\cap d_2^{n-k}=$$ $$= \delta_1\delta_1'  + \dots + \delta_k\delta_k'  + \rho_1 \rho_1^\prime + \dots + \rho_{n-k}\rho_{n-k}^\prime,$$
and this latter expression can equal $n$ only when $\delta_l'=\delta_l$ and $\rho_l'=\rho_l$ for all $l$.  This proves the claim.

To complete the proof of the theorem, calculate
$$\epsilon({\bf x_1y_1},{\bf x_{2}y_{2}}) \cdot \epsilon({\bf x_1y_1},{\bf x_{3}y_{3}})=  c_1^1\cdot c_2^1 + \dots +  c_1^i\cdot c_2^i + d_1^1\cdot d_2^1 + \dots +  d_1^{n-i}\cdot d_2^{n-i} =  i-n.$$
Since $f_*$ preserves the Seifert form this should be equal to
$$\epsilon(\sigma({\bf x_1y_1}),\sigma({\bf x_{2}y_{2}})) \cdot \epsilon(\sigma({\bf x_1y_1}),\sigma({\bf x_{3}y_{3}}))=$$ $$ =\delta_1\delta_1' c_1^1\cdot c_2^1 + \dots + \delta_k\delta_k' c_1^k\cdot c_2^k + \rho_1\rho_1' d_1^1\cdot d_2^1 + \dots + \rho_{n-k}\rho_{n-k}' d_1^{n-k}\cdot d_2^{n-k} =  k-n.$$
Since $i\ne k$, the proof is complete.
\end{proof}

  \section{Directions for Future Research}

As our intent was to introduce the tools necessary for the study of Seifert surfaces through sutured Floer homology, there are many directions one could pursue and questions left unanswered.

One interesting aspect of our particular examples was that the sutured Floer homology groups contained no more information than their Euler characteristic.  Indeed, while the Euler characteristic of sutured Floer homology was not a previously studied invariant, it is classical in the sense that it is an appropriate version of Turaev torsion. On the one hand, this tells us that there is a useful invariant of Seifert surfaces which exists independent of Floer homology (and is easier to compute).  On the other, it begs the question to find examples where the full power of the Floer homology is needed.  Of course taking the connected sum of our examples with the unique minimal genus Seifert surface of a Whitehead doubled knot will produce examples of Seifert surfaces distinguished by sutured Floer homology groups whose Euler characteristic is trivial, but presumably these case could still be handled by ``classical" techniques.

In another direction, we have a  computable invariant which can presumably be used to start the isotopy classification of minimal genus Seifert surfaces for knots in the tables (recall, again, that the previous classifications were for the somewhat less natural notion of strong equivalence).  Since little appears to be known generally about this  type of equivalence, filling in the tables may be a useful first step.

Finally, in analogy with the \os \ knot invariants, it seems reasonable to ask if there are  geometric or combinatorial properties of Seifert surfaces (perhaps having an alternating projection) which constrain the sutured Floer invariants.


\bibliographystyle{amsplain} \bibliography{topology,mybib}
\end{document}